\documentclass[12pt,a4paper,leqno]{amsart}

\usepackage[latin1]{inputenc}
\usepackage[T1]{fontenc}
\usepackage{amsfonts}
\usepackage{amsmath}
\usepackage{amssymb}
\usepackage{eurosym}
\usepackage{mathrsfs}
\usepackage{palatino}

\newcommand{\R}{\mathbb{R}}
\newcommand{\C}{\mathbb{C}}

\newcommand{\Z}{\mathbb{Z}}

\numberwithin{equation}{section}

\makeatletter
  \let\c@subsection\c@equation
\makeatother

\swapnumbers
\theoremstyle{plain}
\newtheorem{thm}[equation]{Theorem}
\newtheorem{lem}[equation]{Lemma}
\newtheorem{prop}[equation]{Proposition}

\theoremstyle{definition}
\newtheorem{defn}[equation]{Definition}

\newtheorem{exmp}[equation]{Example}

\theoremstyle{remark}
\newtheorem{rem}[equation]{Remark}

\numberwithin{equation}{section}

\title{Vector-valued non-homogeneous $Tb$ theorem on metric measure spaces}
\author{Henri Martikainen}\thanks{The author is supported by the Academy of Finland through the project ``$L^p$ methods in harmonic analysis''. The paper is part of the author's doctoral thesis project
written under the supervision of Academy research fellow Tuomas Hytönen -- the guidance of whom is gratefully acknowledged.}
\address{Department of Mathematics and Statistics, University of Helsinki, P.O.B. 68, FI-00014 Helsinki, Finland}
\email{henri.martikainen@helsinki.fi}


\begin{document}

\maketitle

\begin{abstract}
We prove a vector-valued non-homogeneous $Tb$ theorem on certain quasimetric spaces equipped with what we call an upper doubling measure. Essentially, we merge
recent techniques from the domain and range side of things, achieving a $Tb$ theorem which is quite general with respect to both of them.
\end{abstract}

\section{Introduction}
In the seminal paper \cite{NTV} by Nazarov, Treil and Volberg, it was already indicated that it should be possible to prove some version of their (Euclidean) non-homogeneous $Tb$ theorem also
in a more abstract metric space setting, just like the well-established homogeneous theory in this generality \cite{DJS}, \cite{Ch}.
A recent paper \cite{HM} by the author and Tuomas Hytönen shows that this is indeed the case: a non-homogeneous $Tb$ theorem
in the general framework of quasimetric spaces equipped with an upper doubling measure (this is a class of measures that encompasses both the power bounded measures, and also, the more classical doubling measures) was proved.
See also \cite{VW2}.

It is natural to seek to extend the generality in the range too (instead of considering only scalar valued operators). These type of developments, just like the regular scalar valued $Tb$ theorems, have a long history (for a discussion of the
origins of the vector-valued $Tb$ theory consult e.g. \cite{Hy2}). In the very recent work \cite{MP}, a UMD-valued $T1$ theorem is established in metric spaces -- however, only with Ahlfors-regular measures $\mu$ (i.e. $\mu(B(x,r)) \sim r^m$).
This assumption seems to be necessary for their method of proof based on rearrangements of dyadic cubes.
In \cite{Hy2} a vector-valued non-homogeneous $Tb$ theorem is proved in the case of the domain being $\R^n$ and the relevant measure $\mu$ being power bounded (that is, $\mu(B(x,r)) \le Cr^m$).

The methods of \cite{Hy2} are already less dependent on the structure of $\R^n$ than much of the earlier vector-valued work, thus foreshadowing the possibility of extending to more general domains.
The goal here is to carefully combine key techniques from the recent developments \cite{HM} and \cite{Hy2}
and obtain a proof of a non-homogeneous $Tb$ theorem, which is simultaneously general with respect to the domain (a metric space), the measure (an upper doubling measure) and the range (a UMD Banach space).
\section{Preliminaries and the main result}
\subsection{Geometrically doubling quasimetric spaces}
A quasimetric space $(X, \rho)$ is geometrically doubling if every open ball $B(x,r) = \{y \in X: \rho(y,x) < r\}$ can be covered by at most $N$ balls of radius $r/2$.
A basic observation is that in a geometrically doubling quasimetric space, a ball $B(x,r)$ can contain the centers $x_i$ of at most $N\alpha^{-n}$ disjoint balls $B(x_i, \alpha r)$ for $\alpha \in (0,1]$.
Instead of working with what we called reqular quasimetrics in \cite{HM}, it will be assumed for added convenience that $\rho = d^{\beta}$ for some metric $d$ and some constant $\beta \ge 1$ (and not
just equivalent to such a power of a metric). Then $d$-balls are $\rho$-balls and even the weak boundedness property works for both type of balls (this was a somewhat of an inconvenience before). This seems
to be general enough to cover many interesting cases.
\subsection{Upper doubling measures}
A Borel measure $\mu$ in some quasimetric space $(X, \rho)$ is called upper doubling if there exists a dominating function $\lambda\colon X \times (0, \infty) \to (0,\infty)$ so that
$r \mapsto \lambda(x,r)$ is non-decreasing, $\lambda(x,2r) \le C_{\lambda}\lambda(x,r)$ and $\mu(B(x,r)) \le \lambda(x,r)$ for all $x \in X$ and $r > 0$. The number $d:=\log_2 C_{\lambda}$ can be thought of as (an upper bound for) a dimension of the measure $\mu$, and it will play a similar role as the quantity denoted by the same symbol in \cite{NTV}.
\subsection{Standard kernels and Calder\'on--Zygmund operators}
Define $\Delta = \{(x,x): x \in X\}$. A standard kernel
is a mapping $K\colon X^2 \setminus \Delta \to \C$ for which we have for some $\alpha > 0$ and $B,C < \infty$ that 
\begin{displaymath}
|K(x,y)| \le B\min\Big(\frac{1}{\lambda(x, \rho(x,y))}, \frac{1}{\lambda(y, \rho(x,y))}\Big), \qquad x \ne y,
\end{displaymath}
\begin{displaymath}
|K(x,y) - K(x',y)| \le B\frac{\rho(x,x')^{\alpha}}{\rho(x,y)^{\alpha}\lambda(x, \rho(x,y))}, \qquad \rho(x,y) \ge C\rho(x,x'),
\end{displaymath}
and
\begin{displaymath}
|K(x,y) - K(x,y')| \le B\frac{\rho(y,y')^{\alpha}}{\rho(x,y)^{\alpha}\lambda(y, \rho(x,y))}, \qquad \rho(x,y) \ge C\rho(y,y').
\end{displaymath}
The smallest admissible $B$ will be denoted by $\|K\|_{CZ_{\alpha}}$; it is understood that the parameter $C$ has been fixed, and it will not be indicated explicitly in this notation.

Let $T\colon f \mapsto Tf$ be a linear operator acting on some functions $f$ (which we shall specify in more detail later). It is called a Calder\'on--Zygmund operator with kernel $K$ if
\begin{displaymath}
Tf(x) = \int_X K(x,y)f(y)\,d\mu(y)
\end{displaymath}
for $x$ outside the support of $f$.
\subsection{Accretivity}
A function $b \in L^{\infty}(\mu)$ is called accretive if Re$\,b \ge a > 0$ almost everywhere. We can also make do with the following weaker form of accretivity:
$|\int_A b\,d\mu| \ge a\mu(A)$ for all Borel sets $A$ which satisfy the condition that $B \subset A \subset CB$ for some ball $B = B(A)$, where $C$ is some large
constant which depends on the quasimetric $\rho$. (One can e.g. take $C=500$ if dealing with metrics).
\subsection{Weak boundedness property}
An operator $T$ is said to satisfy the weak boundedness property if $|\langle T\chi_B, \chi_B\rangle| \le A\mu(\Lambda B)$ for all balls $B$ and for some fixed constants $A > 0$ and $\Lambda > 1$.
Here $\langle\cdot\,,\, \cdot \rangle$ is the bilinear duality $\langle f, g\rangle = \int fg\,d\mu$. Let us denote the smallest admissible constant above by $\|T\|_{WBP_{\Lambda}}$.

In the $Tb$ theorem, the weak boundedness property is demanded from the operator $M_{b_2}TM_{b_1}$, where
$b_1$ and $b_2$ are accretive functions and $M_b \colon f \mapsto bf$.
\subsection{BMO and RBMO}
We say that $f \in L^1_{\textrm{loc}}(\mu)$ belongs to BMO$^p_{\kappa}(\mu)$, if for any ball $B \subset X$ there
exists a constant $f_B$ such that
\begin{displaymath}
\Big( \int_B |f-f_B|^p\,d\mu\Big)^{1/p} \le L\mu(\kappa B)^{1/p},
\end{displaymath}
where the constant $L$ does not depend on $B$.

Let $\varrho > 1$. A function $f \in L^1_{\textrm{loc}}(\mu)$ belongs to RBMO$(\mu)$ if there exists a constant $L$, and for every ball $B$, a constant $f_B$, such that one has
\begin{displaymath}
\int_B |f-f_B| \,d\mu \le L\mu(\varrho B),
\end{displaymath}
and, whenever $B \subset B_1$ are two balls,
\begin{displaymath}
|f_B - f_{B_1}| \le L\Big(1 + \int_{2B_1 \setminus B} \frac{1}{\lambda(c_B, \rho(x,c_B))}\,d\mu(x)\Big).
\end{displaymath}
We do not demand that $f_B$ be the average $\langle f \rangle_B = \frac{1}{\mu(B)}\int_B f\,d\mu$, and this is actually important in the RBMO$(\mu)$-condition. The useful thing here is that the space RBMO$(\mu)$ is independent of the choice of parameter $\varrho > 1$ and satisfies the John--Nirenberg inequality. For these results in our setting, see \cite{Hy1}.
The norms in these spaces are defined in the obvious way as the best constant $L$.
\subsection{UMD Banach spaces}
A Banach space $Y$ is said to satisfy the UMD property if there holds that
\begin{displaymath}
\Big\| \sum_{k=1}^n \epsilon_k d_k \Big\|_{L^p(\Omega, Y)} \le C\Big\| \sum_{k=1}^n d_k \Big\|_{L^p(\Omega, Y)}
\end{displaymath}
whenever $(d_k)_{k=1}^n$ is a martingale difference sequence in $L^p(\Omega, Y)$ and $\epsilon_k = \pm 1$ are constants. This property does not depend on the parameter
$1 < p < \infty$ in any way.
\subsection{Vinogradov notation and implicit constants}
The notation $f \lesssim g$ is used synonymously with $f \le Gg$ for some constant $G$. We also use $f \sim g$ if $f \lesssim g \lesssim f$. The dependence on the various parameters
should be somewhat clear, but basically $G$ may depend on the various constants of the avove definitions, and on an auxiliary parameter $r$ (which is eventually fixed to depend on the above
parameters only).

\medskip

We now state our main theorem.

\begin{thm}\label{maintheorem}
Let $(X,\rho)$ be a geometrically doubling quasimetric space so that $\rho = d^{\beta}$ for some metric $d$ and $\beta \ge 1$, and assume that this space is equipped with an upper doubling measure $\mu$.
Let $Y$ be a UMD space and $1 < p < \infty$. 
Let $T$ be an $L^p(X,Y)$-bounded Calder\'on--Zygmund operator with a standard kernel $K$,
$b_1$ and $b_2$ be two accretive functions, $\alpha >0 $ and $\kappa,\Lambda > 1$. Then
\begin{equation*}
  \|T\|\lesssim \|Tb_1\|_{BMO^1_{\kappa}(\mu)}+\|T^*b_2\|_{BMO^1_{\kappa}(\mu)}
    +\|M_{b_2}TM_{b_1}\|_{WBP_{\Lambda}}+\|K\|_{CZ_{\alpha}},
\end{equation*}
where the first three terms on the right are in turn dominated by $\|T\|$. Here, of course, $\|T\| = \|T\|_{L^p(X,Y) \to L^p(X,Y)}$.
\end{thm}
Note that it suffices to prove the theorem in the case $\beta = 1$, that is, we are working in an honest metric space $(X,d)$ from now on. We give an example before proceeding with the proof
of the theorem.
\begin{exmp}
In \cite[chapter 12]{HM} we gave an example related to the paper \cite{VW}, and there the application was in a situation where the measure in question was genuinely upper doubling (the doubling theory or the theory of power bounded measures
would not have sufficed), and the space was a quasimetric one (so it really was non-homogeneous theory on metric spaces).

Now we give an example which is actually in the homogeneous situation, but as the domain is a metric space and the range is a general UMD space, this seems not to follow from the previous works.
Also, it goes to show that it is convenient to get this doubling theory as a byproduct of the upper doubling theory.

The example we have in mind is the boundedness of the classical Cauchy--Szegö projection as a UMD-valued operator (this question was asked by Tao Mei through a private communication with Tuomas Hytönen, and Mei
had solved this question in the special case when the range space $Y$ is a so-called non-commutative $L^p$ space).
The setting is the Heisenberg group $\mathbb{H}^n$, which is identified with $\R^{2n+1}$, and is a non-abelian group
where the group operation is given by
\begin{displaymath}
x \cdot y = (x_1 + y_1, \ldots, x_{2n} + y_{2n}, x_{2n+1} + y_{2n+1} - 2 \sum_{j=1}^n (x_jy_{j+n} - x_{j+n}y_n)).
\end{displaymath}
The metric is given by
\begin{displaymath}
d(x,y) = \|x^{-1} \cdot y\|
\end{displaymath}
where
\begin{displaymath}
\|x\| = (\|(x_1, \ldots, x_{2n})\|_{\R^{2n}}^4 + x_{2n+1}^2)^{1/4}.
\end{displaymath}
One can also write $x = [\xi,t] \in \mathbb{H}^n = \mathbb{C}^n \times \R$.
We use the Haar measure for $\mathbb{H}^n$ (this is just the Euclidean Lebesgue measure $d\xi dt$ on $\mathbb{C}^n \times \R$). Now
$\lambda(x,r) = Cr^{2n+2}$ for some appropriate constant $C$.

Using the above notation $x = [\xi,t]$, let $K(x) = C(t + i|\xi|)^{-n-1}$. Set
$K(x,y) = K(y^{-1} \cdot  x)$ for $x \ne y$ (i.e. $y^{-1} \cdot x \ne 0$).
The Cauchy--Szegö projection $C$ is an $L^2$-bounded operator of the form
\begin{displaymath}
Cf(x) = \int_{\mathbb{H}^n} K(x,y)f(y)\,dy.
\end{displaymath}
See e.g. \cite{Stein} for a more exhaustive treatment of the Cauchy--Szegö projection.

Clearly the standard kernel estimates known for $K$ are precisely the same as demanded by our theory with our chosen $\lambda$. Thus, as $C$ is a Calder\'on--Zygmund operator which is bounded as a scalar-valued operator
(and thus satisfies the BMO conditions with e.g. $b_1 = b_2 = 1$ and the weak boundedness property), we have by our above $Tb$ (or $T1$ in this case) theorem that $T$ is a bounded operator
$L^p(\mathbb{H}^n, Y) \to L^p(\mathbb{H}^n, Y)$ for every UMD space $Y$ and for every index $p \in (1, \infty)$.
\end{exmp}

\section{John--Nirenberg theorem for $Tb_1$}
In \cite{HM} it was assumed that $Tb_1, T^*b_2 \in \textrm{BMO}^2_{\kappa}(\mu)$ (and this is natural enough for the $L^2$ theory)
so there it was not necessary to deal with the contents of this chapter. However, now that we are directly doing $L^p$ theory, it seems to be more important to
prove that $Tb_1 \in \textrm{BMO}^1_{\kappa}(\mu) \Longrightarrow Tb_1 \in \textrm{RBMO}(\mu) \Longrightarrow Tb_1 \in \textrm{BMO}^q_{\kappa}(\mu)$ for all $1 < q < \infty$ (and similarly for $T^*b_2$). Indeed, otherwise we would need
to assume a priori that $Tb_1, T^*b_2 \in \bigcap_{1 < q < \infty} \textrm{BMO}^q_{\kappa}(\mu)$. This reduction is known in the Euclidean setting with a power bounded measure (see \cite{NTV}). We now work out the details in our setting.
However, only one key lemma really requires some modifications from the proof found in \cite{NTV}, and so we only sketch the other parts of the argument.
See also \cite{Hy1}, where the details of the RBMO$(\mu)$ theory, especially the John--Nirenberg inequality, are worked out in our setting.
\begin{lem}
Consider some fixed ball $B = B(c_B,r_B)$. There exists $R_B \in [r_B, 1.2r_B]$ so that
\begin{displaymath}
\mu(\{x \in X: R_B - r_Bs < d(x,c_B) < R_B + r_Bs\}) \lesssim s\mu(B(c_B,3r_B))
\end{displaymath}
for all $s \in [0, 1.5]$.
\end{lem}
\begin{proof}
See \cite[p. 184]{NTV}.
\end{proof}
\begin{lem}\label{regint}
If $B = B(c_B,r_B)$ is a ball and $R_B$ is a related regularized radius as in the previous lemma, then it holds that
{\setlength\arraycolsep{2pt}
\begin{eqnarray*}
\int_{B(c_B, R_B)} \int_{B(c_B, 3r_B) \setminus B(c_B, R_B)} \!\!\!\!\!\!\!\!\!\! |K(x,y)|\,d\mu(y)\,d\mu(x) & \lesssim &  \mu(B(c_B,R_B))^{1/2}\mu(B(c_B, 3r_B))^{1/2} \\
& \le & \mu(B(c_B,3r_B)).
\end{eqnarray*}
}
\end{lem}
\begin{proof}
Consider $f(x) = \int_{B(c_B, 3r_B) \setminus B(c_B, R_B)} |K(x,y)|\,d\mu(y)$, $x \in B(c_B, R_B)$. Fix $x \in B(c_B, R_B)$ for the moment and note that we have for all $y \in B(c_B, 3r_B) \setminus B(c_B, R_B)$ that
$d(x,y) \le R_B + 3r_B \le 4.2r_B < 5r_B$ and $d(x,y) \ge d(y,c_B) - d(x,c_B) \ge R_B - d(x,c_B)$. We temporarily set $h = R_B - d(x,c_B)$ for this fixed $x$ and estimate
{\setlength\arraycolsep{2pt}
\begin{eqnarray*}
f(x) & \lesssim & \int_{h \le d(x,y) < 5r_B} \frac{d\mu(y)}{\lambda(x, d(x,y))} \\
& \le & \sum_{1 \le j < \log_2(10r_B/h)} \int_{2^{j-1}h\le d(x,y) < 2^jh} \frac{d\mu(y)}{\lambda(x, d(x,y))} \\
& \le & \sum_{1 \le j < \log_2(10r_B/h)} \frac{\mu(B(x, 2^jh))}{\lambda(x, 2^{j-1}h)} \\
& \lesssim & \log(10r_B/h) \\
& = & \log\Big( \frac{10r_B}{R_B - d(x,c_B)}\Big).
\end{eqnarray*}
}
This implies through Hölder's inequality that
\begin{displaymath}
\int_{B(c_B, R_B)} \!\!\! f(x)\,d\mu(x) \lesssim \mu(B(c_B, R_B))^{1/2}\Big(\int_{B(c_B, R_B)}\!\! \Big[\log\Big( \frac{10r_B}{R_B - d(x,c_B)}\Big)\Big]^2\,d\mu(x)\Big)^{1/2}.
\end{displaymath}
We then continue to note that
\begin{displaymath}
\int_{B(c_B, R_B)} \Big[\log\Big( \frac{10r_B}{R_B - d(x,c_B)}\Big)\Big]^2\,d\mu(x) 
\end{displaymath}
equals
\begin{displaymath}
\int_0^{\infty} \mu\Big(\Big\{x \in B(c_B, R_B):  \Big[\log\Big( \frac{10r_B}{R_B - d(x,c_B)}\Big)\Big]^2 > t\Big\}\Big)\,dt,
\end{displaymath}
which in turn equals
\begin{displaymath}
\int_0^{\infty} \mu(\{x: R_B - 10r_Be^{-\sqrt{t}} < d(x,c_B) < R_B\})\,dt = \Big[\log \Big(\frac{10r_B}{R_B}\Big)\Big]^2\mu(B(c_B,R_B))
\end{displaymath}
\begin{displaymath}
\qquad \qquad \qquad \qquad + \int_{[\log(10r_B/R_B)]^2}^{\infty} \mu(\{x: R_B - 10r_Be^{-\sqrt{t}} < d(x,c_B) < R_B\})\,dt.
\end{displaymath}
Note that $\int_0^{\infty} e^{-\sqrt{t}}\,dt = 2$ and use the previous lemma with $s = 10e^{-\sqrt{t}} \le R_B/r_B \le 1.2 < 1.5$ for $t \ge [\log(10r_B/R_B)]^2$ to get that
\begin{displaymath}
\int_{[\log(10r_B/R_B)]^2}^{\infty} \mu(\{x: R_B - 10r_Be^{-\sqrt{t}} < d(x,c_B) < R_B\})\,dt \lesssim \mu(B(c_B,3r_B)).
\end{displaymath}
This yields the claim.
\end{proof}
\begin{thm}
Under the assumptions of Theorem \ref{maintheorem},
there holds that $Tb_1 \in \textrm{RBMO}(\mu)$, especially $Tb_1 \in \bigcap_{1 < q < \infty} \textrm{BMO}^q_{\kappa_1}(\mu)$ for any $\kappa_1 > 1$.
\end{thm}
\begin{proof}
It suffices to prove that for every ball $B$ the function $T(\chi_{10B}b_1)$ satisfies the defining properties of the RBMO$(\mu)$ space for all the balls that are subset of $B$, and in such a way that the RBMO norm does not
depend on $B$. To see that this suffices, note that $|Tb_1-T(\chi_{10B}b_1)| \lesssim 1$ on $B$ for all balls $B$. The hardest part of the remaining proof consists of proving that
\begin{displaymath}
\int_B |T(\chi_{2B}b_1)|\,d\mu \lesssim \mu(\eta B)
\end{displaymath}
for $\eta = \max(2\kappa, 2\Lambda, 3)$ (the rest of the proof unfolds naturally).
This inequality follows from duality using the assumption $Tb_1 \in \textrm{BMO}^1_{\kappa}(\mu)$, the weak boundedness property, the previous lemma and the fact
that $b_1$ is accretive. These details follow as in \cite[chapter 2]{NTV}.
\end{proof}
\section{Random dyadic systems and good/bad cubes}
One feature of the proof in \cite{Hy2} is that one basically takes all the cubes to be good in the various summations -- this is in contrast with the proof in \cite{HM} where things were usually summed
so that the bigger cubes are arbitrary but the smaller cubes from the other grid were assumed to be good. This modification seems to be particularly useful when dealing with certain paraproducts in these general UMD spaces.

This leads us to fiddle with our randomization from \cite{HM} quite a bit. We shall make the randomization so that there is no removal procedure involved (unlike in \cite{HM}) -- then a certain index set may serve as a fixed
reference set more conveniently. Such a modification will also be used in a future paper by T. Hytönen and A. Kairema, and the author learned about the details of this modification from them through a private
communication.

Furthermore, we will change the definition of a good cube to be such that given a cube $Q$ its change to be good does not depend on the smaller cubes $R$ with $\ell(R) \le \ell(Q)$.
Related to this we shall also make a minor tweak to our half-open cubes from \cite{HM} (to get a better dependence on the randomized dyadic points).
Finally, we add a layer of artificial badness so that $\mathbb{P}(Q \textrm{ is good})$ does not depend on the particular choice of the cube $Q$.

Let us get to the details. Let $\delta = 1/1000$. We recall from \cite{HM} (see also \cite{Ch} for the original construction)
that given a collection of points $x^k_{\alpha}$ such that $d(x^k_{\alpha}, x^k_{\beta}) \ge \delta^k/8$ for all $\alpha \ne \beta$ and
$\min_{\alpha} d(x,x^k_{\alpha}) < 4\delta^k$, we may define a certain transitive relation $\le_{\mathcal{D}}$ between these points, and then there exists sets $Q^k_{\alpha}$ (we call these half-open dyadic cubes)
so that for every $k \in \Z$ we have
\begin{displaymath}
X = \bigcup_{\alpha} Q^k_{\alpha},
\end{displaymath}
for every $k \in \Z$ and $\ell \ge k$ it holds that either $Q^k_{\alpha} \cap Q^{\ell}_{\beta} = \emptyset$ or $Q^{\ell}_{\beta} \subset Q^k_{\alpha}$, and for every
$\ell \ge k$ we have
\begin{displaymath}
Q^k_{\alpha} = \bigcup_{\beta: (\ell, \beta) \le_{\mathcal{D}} (k, \alpha)} Q^{\ell}_{\beta}.
\end{displaymath}
This set of cubes is denoted by $\mathcal{D} = \{Q^k_{\alpha}\}$. Moreover, these cubes satisfy that $d(Q^k_{\alpha}) < C_0\delta^k$ and $B(x^k_{\alpha}, C_1\delta^k) \subset Q^k_{\alpha}$ for
$C_0 = 10$ and $C_1 = 1/100$. We denote $\ell(Q^k_{\alpha}) = \delta^k$.

We now fix some large natural number $k_0$ the value of which will be specified more carefully in the next chapter. The tweak we make to the construction of the above cubes is simple: we follow the construction in \cite[chapter 4]{HM}
except that in the proof of \cite[Theorem 4.4]{HM} we make the construction so that $k=0$, $k < 0$ and $k > 0$ are replaced by $k=k_0$, $k<k_0$ and $k > k_0$ respectively. The point is that the original dyadic cubes of M. Christ
(which may not cover the whole space unlike these half-open ones) have a better dependence on the centers $x^k_{\alpha}$ than the half-open cubes. After the modification, however, we have this more favourable dependence at least for
all the cubes of generations $k \le k_0$. Indeed, now a cube $Q^k_{\alpha}$, where $k \le k_0$, depends only on the centers $x^{\ell}_{\beta}$ for $\ell \ge k$. All the properties stated above remain valid, of course.

We explain the modified randomization now. One starts by fixing once and for all the reference points $z^k_{\alpha}$ satisfying
$d(z^k_{\alpha}, z^k_{\beta}) \ge \delta^k$ for all $\alpha \ne \beta$ and $\min_{\alpha} d(x,z^k_{\alpha}) < \delta^k$. We also fix one
relation $\le$ related to these points. We say that $(k, \alpha)$ and $(k, \beta)$ conflict if $d(z^{k+1}_{\gamma}, z^{k+1}_{\sigma}) < \delta^k/4$
for some $(k+1, \gamma) \le (k, \alpha)$ and $(k+1, \sigma) \le (k, \beta)$. Let $I(k, \alpha)$ be the set of pairs $(k, \beta)$ conflicting with $(k, \alpha)$.
Note that $\#I(k, \alpha) \lesssim 1$ as $X$ is geometrically doubling. We now earmark the points $z^k_{\alpha}$ (or the indices $(k, \alpha)$). To this end, fix some $L > \max_{(k, \alpha)} \#I(k, \alpha)$. Let $k \in \Z$ be given.
We inductively tag $z^k_{\alpha}$ by associating it with the smallest number $i \in \{1, \ldots, L\}$ having the feature that no $(k,\beta) \in I(k, \alpha)$ that has already been tagged
is associated with this number (recall that $\alpha$ always varies only over some countable set).

We now associate to each $(k, \alpha)$ a new point $x^k_{\alpha}$ in a random way. First one randomly chooses $i \in \{1, \ldots, L\}$ (uniform distribution, of course).
If $(k, \alpha)$ happens to be earmarked with the number $i$, we set $x^k_{\alpha} = z^{k+1}_{\beta}$ for some $(k+1, \beta) \le (k, \alpha)$, and the choice is made using uniform probability
(there are only boundedly many indices $(k+1, \beta) \le (k, \alpha)$). If $(k, \alpha)$ is not tagged with the number $i$, we set $x^k_{\alpha} = z^{k+1}_{\beta}$ for
some $(k+1, \beta)$ for which it holds that $d(z^k_{\alpha}, z^{k+1}_{\beta}) < \delta^{k+1}$ (there is always at least one such point available by construction). To summarize, for $i$-tagged indices we
randomly choose any $z^{k+1}_{\beta}$ for which $(k+1, \beta) \le (k, \alpha)$ and for the rest we choose some special $z^{k+1}_{\beta}$ which is particularly close to $z^k_{\alpha}$.
This is done independently on all levels $k \in \Z$. The idea of using this tagging as a way to avoid the removal procedure used in \cite{HM} is by T. Hytönen and A. Kairema.

The result is some new set of points $x^k_{\alpha}$, which readily qualify as new dyadic points (that is, $d(x^k_{\alpha}, x^k_{\beta}) \ge \delta^k/8$ for all $\alpha \ne \beta$
and $\min_{\alpha} d(x,x^k_{\alpha}) < 4\delta^k$ (with some better constants even)). This is an easy consequence of the construction, and we omit the details. Also evident is the fact
that $\mathbb{P}(z^{k+1}_{\beta} = x^k_{\alpha}) \ge \pi_0 > 0$ for some absolute constant $\pi_0$ if $(k+1, \beta) \le (k, \alpha)$ (this needed an extra argument with the randomization used in \cite{HM}).
Now the same proof as in \cite[Lemma 10.1]{HM} also gives us the same result with this modified randomization. That is, we have:
\begin{lem}\label{boundarylemma}
For any fixed $x \in X$ and $k \in \Z$, it holds
\begin{displaymath}
\mathbb{P}(x \in \delta_{Q^k_{\alpha}} \textrm{ for some }\alpha) \lesssim \epsilon^{\eta}
\end{displaymath}
for some $\eta > 0$. Here $\delta_{Q^k_{\alpha}} = \{x: d(x,Q^k_{\alpha}) \le \epsilon \ell(Q^k_{\alpha}) \textrm{ and } d(x, X \setminus Q^k_{\alpha}) \le \epsilon \ell(Q^k_{\alpha})\}$.
\end{lem}

We shall now modify the notion of goodness. Here we are given two dyadic systems of cubes $\mathcal{D} = \{Q^k_{\alpha}\}$ and $\mathcal{D}' = \{R^k_{\alpha}\}$ as always.
This amounts to randomly producing two sets of new dyadic points $(x^k_{\alpha})$ and $(y^k_{\alpha})$ using the above procedure and then choosing (following certain established rules but somewhat arbitrarily)
some relations $\le_{\mathcal{D}}$ and $\le_{\mathcal{D}'}$ related to the systems $(x^k_{\alpha})$ and $(y^k_{\alpha})$ respectively. Indeed, this information generates the families of cubes
$\mathcal{D} = \{Q^k_{\alpha}\}$ and $\mathcal{D}' = \{R^k_{\alpha}\}$. Set
\begin{displaymath}
\gamma := \frac{\alpha}{2(\alpha + d)},
\end{displaymath}
where we recall that $d := \log_2 C_{\lambda}$ in our setting.
\begin{defn}
We say that $Q^k_{\alpha} \in \mathcal{D}$ is geometrically $\mathcal{D}'$-bad, if there exists $(k-s, \beta) \ne (k-s, \gamma)$ for some $s \ge r$ so that
for some $(k-1, \eta) \le_{\mathcal{D}'} (k-s, \beta)$ and $(k-1, \xi) \le_{\mathcal{D}'} (k-s, \gamma)$ we have
$d(x^k_{\alpha}, y^{k-1}_{\eta}) \le \delta^{\gamma k}\delta^{(1-\gamma)(k-s)}$ and $d(x^k_{\alpha}, y^{k-1}_{\xi}) \le \delta^{\gamma k}\delta^{(1-\gamma)(k-s)}$. Otherwise
$Q^k_{\alpha}$ is geometrically $\mathcal{D}'$-good.
\end{defn}
Here the new feature is that with this definition the badness of a cube $Q^k_{\alpha}$ depends only on the centers of generations $\ell < k$ of the other system.
Let us then explain why this is still pretty close to the definition given in \cite{HM}. Note that $\delta^k = \delta^{(1-\gamma)s} \cdot \delta^{\gamma k}\delta^{(1-\gamma)(k-s)}$ and
$\delta^{(1-\gamma)s} \le \delta^{(1-\gamma)r} < 10^{-5}$ (as $r$ is fixed to be big enough). Suppose $Q^k_{\alpha}$ is good and $s \ge r$. We have that $x^k_{\alpha} \in R^{k-1}_{\eta} \subset
R^{k-s}_{\beta}$ for some unique $(k-1, \eta) \le_{\mathcal{D}'} (k-s, \beta)$. Now $d(x^k_{\alpha}, y^{k-1}_{\eta}) < 10\delta^{k-1} = 10^4 \delta^k < \delta^{\gamma k}\delta^{(1-\gamma)(k-s)}$.
Suppose (aiming for a contradiction) that we would have $d(x^k_{\alpha}, X \setminus R^{k-s}_{\beta}) < (3/4)\delta^{\gamma k}\delta^{(1-\gamma)(k-s)}$. Then we would have
for some $z \in X \setminus R^{k-s}_{\beta}$ that $d(x^k_{\alpha}, z) \le (3/4)\delta^{\gamma k}\delta^{(1-\gamma)(k-s)}$. But then $z \in R^{k-1}_{\xi} \subset R^{k-s}_{\gamma}$
for some $(k-1, \xi) \le_{\mathcal{D}'} (k-s, \gamma) \ne (k-s, \beta)$, and
\begin{displaymath}
d(x^k_{\alpha}, y^{k-1}_{\xi}) \le d(x^k_{\alpha}, z) + d(z, y^{k-1}_{\xi}) \le [3/4 + 10^{-1}]\delta^{\gamma k}\delta^{(1-\gamma)(k-s)} < \delta^{\gamma k}\delta^{(1-\gamma)(k-s)}
\end{displaymath}
contradicting the goodness of $Q^k_{\alpha}$. So we must have
\begin{align*}
d(Q^k_{\alpha}, X \setminus R^{k-s}_{\beta}) &\ge d(x^k_{\alpha}, X \setminus R^{k-s}_{\beta}) - 10\delta^k \\
&\ge [3/4 - 10^{-4}]\delta^{\gamma k}\delta^{(1-\gamma)(k-s)} \ge 2^{-1}\delta^{\gamma k}\delta^{(1-\gamma)(k-s)}.
\end{align*}
Thus also $d(Q^k_{\alpha}, R^{k-s}_{\gamma}) \ge 2^{-1}\delta^{\gamma k}\delta^{(1-\gamma)(k-s)}$ for every $\gamma \ne \beta$. We record these easy observations as a lemma.
\begin{lem}
If $Q \in \mathcal{D}$ is geometrically $\mathcal{D}'$-good, then for every $R \in \mathcal{D}'$ for which $\ell(Q) \le \delta^r\ell(R)$ we have
either $d(Q, R) \gtrsim \ell(Q)^{\gamma}\ell(R)^{1-\gamma}$ or $d(Q, X \setminus R) \gtrsim \ell(Q)^{\gamma}\ell(R)^{1-\gamma}$.
\end{lem}
If $Q^k_{\alpha}$ is bad, then the definition demands that for some $s \ge r$ we have that $x^k_{\alpha} \in R^{k-s} \in \mathcal{D}'$ so that $d(x^k_{\alpha}, X \setminus R^{k-s}) \le \delta^{\gamma k}\delta^{(1-\gamma)(k-s)}
= \delta^{\gamma s} \delta^{k-s} = \delta^{\gamma s} \ell(R^{k-s})$. Lemma \ref{boundarylemma} with $\epsilon = \delta^{\gamma s}$ then yields that
\begin{displaymath}
\mathbb{P}(Q^k_{\alpha} \textrm{ is geometrically }\mathcal{D}'\textrm{-bad}) \lesssim \sum_{s=r}^{\infty} (\delta^{\gamma \eta})^s \lesssim \delta^{r\gamma \eta}.
\end{displaymath}
We have proved the following.
\begin{lem}\label{badarerare}
For a fixed $Q \in \mathcal{D}$ we have under the random choice of the $\mathcal{D}'$-grid that
\begin{displaymath}
\mathbb{P}(Q \textrm{ is geometrically }\mathcal{D}'\textrm{-bad}) \lesssim \delta^{r\gamma \eta}.
\end{displaymath}
\end{lem}
We still need to achieve the effect that $\mathbb{P}(Q \textrm{ is good})$ would not depend on the particular choice of the cube $Q$ (in $\R^n$ this followed from symmetry, see \cite{Hy2}).
There seems to be no obvious reason why this should be the case already,
so we will force this by understanding goodness in a stronger sense: a cube is good if it is geometrically good and pseudogood -- a notion to be defined.

Define $\pi_{x^k_{\alpha}} = \mathbb{P}(Q^k_{\alpha} \textrm{ is geometrically }\mathcal{D}'\textrm{-good})$. Note that under the random choice of the other grid $\mathcal{D}'$, this really depends only on the center $x^k_{\alpha}$ of
$Q^k_{\alpha}$. Set $\pi_{\textrm{good}} = 1 - C\delta^{r\gamma \eta}$ so that always $\pi_{x^k_{\alpha}} \ge \pi_{\textrm{good}}$. Set $Z(t^k_{\alpha}, x^k_{\alpha}) = 1$, if $0 \le t^k_{\alpha} \le \pi_{\textrm{good}}/\pi_{x^k_{\alpha}}$, and
$Z(t^k_{\alpha}, x^k_{\alpha}) = 0$, if $1 \ge t^k_{\alpha} > \pi_{\textrm{good}}/\pi_{x^k_{\alpha}}$. Now $\mathbb{P}(Z(t^k_{\alpha}, x^k_{\alpha}) = 1\,|\, x^k_{\alpha}) = \pi_{\textrm{good}}/\pi_{x^k_{\alpha}}$ using the Lebesgue
measure on the interval $[0,1]$. We say that $Q^k_{\alpha}$ is pseudogood if $Z(t^k_{\alpha}, x^k_{\alpha}) = 1$, and $\mathcal{D}'$-good if it is geometrically $\mathcal{D}'$-good and pseudogood. If one considers the grid
$\mathcal{D}$ to be fixed, then under the random choice of the pseudogoodness parameters and the grid $\mathcal{D}'$, we have by independence that
$\mathbb{P}(Q^k_{\alpha} \textrm{ is } \mathcal{D}'\textrm{-good}) = \pi_{\textrm{good}}$ for every $Q^k_{\alpha} \in \mathcal{D}$. We use analogous random variables $W(u^k_{\alpha}, y^k_{\alpha})$ to determine
the pseudogoodness status of a cube in the grid $\mathcal{D}'$, and then the $\mathcal{D}$-goodness is also similarly defined.

Basically all these modification were done to prove the following analogue of \cite[Lemma 5.2]{Hy2} with our randomized systems of metric dyadic cubes. This enables us to later establish that
a certain paraproduct is bounded following the strategy used in \cite{Hy2}.

First a few comments. In the following chapter we shall introduce two fixed functions $f$ and $g$, and their martingale difference decompositions using Haar functions.
The aim is then to control a certain average (\ref{normest}). The details of this are not important for the next lemma, except for the fact that looking at that particular sum one sees that it is enough to sum over some fixed finite
index set $(k, \alpha)$ (because the functions have bounded support, the space is geometrically doubling, and cubes of only finitely many generations are needed). Thus, we assume that such is the case in the next lemma also.
This enables us to move $\mathbb{E}$ in and out the summation freely (see the proof). Also, $\varphi(Q,R)$ is an $L^1$-function of cubes $Q$ and $R$ and their children -- basically in the only application of this lemma we take
$\varphi(Q,R) = \langle g, \psi_R \rangle \langle b_2, T(b_1\varphi_Q)\rangle\langle \psi_R\rangle_Q \langle \varphi_Q, f \rangle$ (see the chapters 5 and 8).
\begin{lem}\label{goodtonogood}
We have that
\begin{displaymath}
(1 - C\delta^{r\gamma \eta})\mathbb{E} \sum_{R \in \mathcal{D}'} \mathop{\sum_{Q \in \mathcal{D}_{\textrm{good}}}}_{\delta^{k_0} < \ell(Q) \le \ell(R)} \varphi(Q,R)
= \mathbb{E} \sum_{R \in \mathcal{D}_{\textrm{good}}'} \mathop{\sum_{Q \in \mathcal{D}_{\textrm{good}}}}_{\delta^{k_0} < \ell(Q) \le \ell(R)} \varphi(Q,R),
\end{displaymath}
where the grid $\mathcal{D}'$ is fixed (so a set of points $(y^k_{\alpha})$ is fixed) and we average over every other random quantity $((x^k_{\alpha})$, $(t^k_{\alpha})$, $(u^k_{\alpha}))$.
\end{lem}
\begin{proof}
We start by recalling the dependencies (remember that the points $(y^m_{\gamma})$ are fixed). The goodness of a cube $R^m_{\gamma} \in \mathcal{D}'$ depends on the points
$x^k_{\alpha}$ for which $k < m$ and on $u^m_{\gamma}$. The goodness of a cube $Q^k_{\alpha} \in \mathcal{D}$ depends on $x^k_{\alpha}$ and $t^k_{\alpha}$. As sets,
$Q^k_{\alpha}$ and its children depend on the centers $x^{\ell}_{\beta}$ for which $\ell \ge k$ (and this is because of the restriction $\delta^k = \ell(Q^k_{\alpha}) > \delta^{k_0}$ which says $k < k_0$). 

Note that $\pi_{\textrm{good}} = \mathbb{P}(R \in \mathcal{D}'_{\textrm{good}}) = \mathbb{E}(\chi_{\textrm{good}}(R))$ for every $R \in \mathcal{D}'$.
Thus, we have
\begin{align*}
\pi_{\textrm{good}} \mathbb{E} \sum_{R \in \mathcal{D}'} \mathop{\sum_{Q \in \mathcal{D}_{\textrm{good}}}}_{\delta^{k_0} < \ell(Q) \le \ell(R)} \varphi(Q,R) &= \pi_{\textrm{good}}
\mathbb{E} \sum_{(m,\gamma)} \mathop{\sum_{(k, \alpha)}}_{m \le k < k_0} \chi_{\textrm{good}}(Q^k_{\alpha}) \varphi(Q^k_{\alpha},R^m_{\gamma}) \\
&= \sum_{(m,\gamma)} \mathop{\sum_{(k, \alpha)}}_{m \le k < k_0} \mathbb{E}(\chi_{\textrm{good}}(R^m_{\gamma})) \mathbb{E}(\chi_{\textrm{good}}(Q^k_{\alpha})\varphi(Q^k_{\alpha},R^m_{\gamma})) \\
&= \sum_{(m,\gamma)} \mathop{\sum_{(k, \alpha)}}_{m \le k < k_0} \mathbb{E}(\chi_{\textrm{good}}(R^m_{\gamma})\chi_{\textrm{good}}(Q^k_{\alpha})\varphi(Q^k_{\alpha},R^m_{\gamma})) \\
&= \mathbb{E} \sum_{R \in \mathcal{D}'_{\textrm{good}}} \mathop{\sum_{Q \in \mathcal{D}_{\textrm{good}}}}_{\delta^{k_0} < \ell(Q) \le \ell(R)} \varphi(Q,R).
\end{align*}

Let us still spell out the details of the above computation (since it is actually surprisingly subtle and depends on all of the modifications made above).
We first removed everything that is random from the summations.
Then we moved the expectation inside the summation (the sum is finite by assumption), and after that we also moved the constant $\pi_{\textrm{good}} = 1 - C\delta^{r\gamma \eta}$ inside the summation
noting then that it equals $\mathbb{E}(\chi_{\textrm{good}}(R^m_{\gamma}))$ with any $(m, \gamma)$. Next we used the product rule of expectations of independent quantities: the random variable
$\chi_{\textrm{good}}(R^m_{\gamma})$ depends on $x^{\ell}_{\beta}$ for $\ell < m$ and on $u^m_{\gamma}$, and the random variable
$\chi_{\textrm{good}}(Q^k_{\alpha})\varphi(Q^k_{\alpha},R^m_{\gamma})$ depends on $x^{\ell}_{\beta}$ for $\ell \ge k \ge m$ and on $t^k_{\alpha}$. Recall also that the points
of different generations are independently chosen. Finally we moved the expectation out and rewrote the summation so that it again contains the random quantities.
\end{proof}
\section{Martingale difference decomposition, Haar functions and the tangent martingale trick}
Let us be given some system of cubes $\{Q^k_{\alpha}\}$ and some accretive function $b$. We set
\begin{align*}
E_k^bf &= \sum_{\alpha} \langle f \rangle_{Q^k_{\alpha}}\langle b \rangle_{Q^k_{\alpha}}^{-1} \chi_{Q^k_{\alpha}}b, \\
E_{Q^k_{\alpha}}^bf &= \chi_{Q^k_{\alpha}}E_k^b f, \\
\Delta_k^bf &= E_{k+1}^bf - E_k^bf, \\
\Delta_{Q^k_{\alpha}}^bf & = \chi_{Q^k_{\alpha}}\Delta_k^b f.
\end{align*}

Consider some cube $Q$. It has subcubes of the next generation $Q_i$, $i=1, \ldots, s(Q)$, where $s(Q) \lesssim 1$.
We set $\hat{Q}_k = \bigcup_{i=k}^{s(Q)} Q_i$, and note that we can always arrange the indexation of the subcubes to be such that $|b(\hat Q_k)| \gtrsim \mu(Q)$ for
every $k=1, \ldots, s(Q)$. Indeed, we can index so that (here $a$ is the accretivity constant of $b$)
\begin{displaymath}
|b(\hat Q_k)| \ge \Big(1-\frac{k-1}{s(Q)}\Big)a\mu(Q) \gtrsim \mu(Q),
\end{displaymath}
and this can proven as \cite[Lemma  4.3]{Hy2}. Note also that trivially $|b(\hat Q_k)| \lesssim \mu(Q)$ (so $|b(\hat Q_k)| \sim \mu(Q)$) and $|b(Q_i)| \sim \mu(Q_i)$.

Now define
\begin{displaymath}
\Delta^b_{Q,u}f = E^b_{Q_u}f + E^b_{\hat{Q}_{u+1}}f - E^b_{\hat{Q}_u}f
\end{displaymath}
also noting that
\begin{displaymath}
\Delta^b_Q f = \sum_{u=1}^{s(Q)-1} \Delta^b_{Q,u}f.
\end{displaymath}
A computation shows that
\begin{displaymath}
\Delta^b_{Q,u}f = b\varphi^b_{Q,u}\langle \varphi^b_{Q,u}, f\rangle,
\end{displaymath}
where we have the adapted Haar functions
\begin{displaymath}
\varphi^b_{Q,u} = \Big( \frac{b(Q_u)b(\hat{Q}_{u+1})}{b(\hat{Q}_u)}\Big)^{1/2} \Big( \frac{\chi_{Q_u}}{b(Q_u)} - \frac{\chi_{\hat{Q}_{u+1}}}{b(\hat{Q}_{u+1})}\Big)
\end{displaymath}
as in \cite{Hy2}. Here we have to interpret $\varphi^b_{Q,u} = 0$ if $\mu(Q_u) = 0$. We also have the non-cancellative adapted Haar function
\begin{displaymath}
\varphi^b_{Q,0}f = b(Q)^{-1/2}\chi_Q
\end{displaymath}
using which we write $E^b_Qf = b\varphi^b_{Q,0}\langle \varphi^b_{Q,0},f\rangle$.

We record the key properties (the last two being only important special cases)
\begin{displaymath}
\int b\varphi^b_{Q,u} \,d\mu = 0,
\end{displaymath}
\begin{displaymath}
|\varphi^b_{Q,u}| \sim \mu(Q_u)^{1/2}\Big(\frac{\chi_{Q_u}}{\mu(Q_u)} + \frac{\chi_{\hat Q_{u+1}}}{\mu(Q)}\Big),
\end{displaymath}
\begin{displaymath}
\|\varphi^b_{Q,u}\|_{L^p(X)} \sim \mu(Q_u)^{1/p-1/2}
\end{displaymath}
and
\begin{displaymath}
\|\varphi^b_{Q,u}\|_{L^1(X)}\|\varphi^b_{Q,u}\|_{L^{\infty}(X)} \sim 1.
\end{displaymath}

Given a dyadic system $\mathcal{D} = \{Q\}$ we can write with any $m$ that
\begin{align*}
f &= \mathop{\sum_{Q \in \mathcal{D}}}_{\ell(Q) \le \delta^m} \Delta_Q^{b_1}f +  \mathop{\sum_{Q \in \mathcal{D}}}_{\ell(Q) = \delta^m} E_Q^{b_1}f \\
&= \mathop{\sum_{Q \in \mathcal{D}}}_{\ell(Q) \le \delta^m} \sum_u b_1\varphi^{b_1}_{Q,u}\langle \varphi^{b_1}_{Q,u}, f\rangle,
\end{align*}
where the $u$ summation runs through $1, \ldots, s(Q)-1$ if $\ell(Q) < \delta^m$, and through $0, 1, \ldots, s(Q)-1$ if $\ell(Q) = \delta^m$. The unconditional convergence
of this in $L^p(X,Y)$ is not at all clear, but it nevertheless follows as in \cite[Proposition 4.1]{Hy2} (note that in that proof certain abstract paraproducts are used, but their
theory is formulated in chapter 3 of \cite{Hy2} in an abstract filtered space which directly applies also in our situation).

Basically the strategy we shall use is the usual one: write the same decomposition for a function $g \in L^{p'}(X,Y^*)$ just using some other grid $\mathcal{D}' = \{R\}$ and the other test function $b_2$, and then decompose
the pairing $\langle g, Tf \rangle$ accordingly. However, Lemma \ref{goodtonogood} has the restriction involving $k_0$ (which we have not yet fixed) and so we somehow need to get into a situation where we do not need to consider
arbitrarily small cubes.

We start by choosing two boundedly supported functions $f \in L^p(X,Y)$ and $g \in L^{p'}(X, Y^*)$ so that $f/b_1$ and $g/b_2$ are Lipschitz, $\|f\|_{L^p(X,Y)} = \|g\|_{L^{p'}(X,Y^*)} = 1$ and $\|T\| \le 2|\langle g, Tf\rangle|$.
Here, of course, $\|T\| =\|T\|_{L^p(X,Y) \to L^p(X,Y)}$. For the fact that Lipschitz functions are dense, see e.g. the proof of \cite[Proposition 3.4]{Hy1}. We now also fix $m$ so that the supports of the functions $f$ and $g$
are contained in some balls $B(x_0, \delta^m)$ and $B(x_1, \delta^m)$ respectively.

Using any two dyadic systems $\mathcal{D}$ and $\mathcal{D}'$ we decompose
\begin{displaymath}
\langle g, Tf \rangle = \langle g - E^{b_2}_{k_0}g, Tf \rangle + \langle E^{b_2}_{k_0}g, T(f - E^{b_1}_{k_0}f)\rangle + \langle E^{b_2}_{k_0}g, T(E^{b_1}_{k_0}f)\rangle,
\end{displaymath}
and then estimate
\begin{align*}
|\langle g, Tf \rangle| \le \|T\|& \|g - E^{b_2}_{k_0}g\|_{L^{p'}(X, Y^*)} \|f\|_{L^p(X,Y)} \\ 
&+ \|T\| \| E^{b_2}_{k_0}g \|_{L^{p'}(X, Y^*)} \| f - E^{b_1}_{k_0}f \|_{L^p(X,Y)} + |\langle E^{b_2}_{k_0}g, T(E^{b_1}_{k_0}f)\rangle|.
\end{align*}
Note that $\| E^{b_2}_{k_0}g \|_{L^{p'}(X, Y^*)} \lesssim \| g \|_{L^{p'}(X, Y^*)} = 1$ so that we get
\begin{displaymath}
|\langle g, Tf \rangle| \le (C(b_2)\| f - E^{b_1}_{k_0}f \|_{L^p(X,Y)}  + \|g - E^{b_2}_{k_0}g\|_{L^{p'}(X, Y^*)})\|T\| + |\langle E^{b_2}_{k_0}g, T(E^{b_1}_{k_0}f)\rangle|.
\end{displaymath}

Next we employ the facts that $f/b_1$ and $g/b_2$ are Lipschitz (with a constant $L$ say). Let $h = f/b_1$. Let $x \in X$ and then let $Q$ denote the unique $\mathcal{D}$-cube of generation $k_0$ containing $x$.
We have that
\begin{align*}
\| E^{b_1}_{k_0} f(x) - f(x)\|_Y &\lesssim \|\langle b_1 \rangle_Q h(x) - \langle b_1 h\rangle_Q \|_Y \\
&\le\frac{1}{\mu(Q)} \int_Q |b_1(z)|\,\|h(z)-h(x)\|_Y\,d\mu(z) \\
&\lesssim Ld(Q) \lesssim L\delta^{k_0}.
\end{align*}
Noting that $\bigcup\{Q:\, Q \in \mathcal{D}_{k_0}, \, Q \cap B(x_0, \delta^m) \ne \emptyset\} \subset B(x_0, 2\delta^m)$ we have that
\begin{displaymath}
\| f - E^{b_1}_{k_0}f \|_{L^p(X,Y)} \lesssim L\lambda(x_0, \delta^m)^{1/p}\delta^{k_0}.
\end{displaymath}
A similar estimate holds for $\|g - E^{b_2}_{k_0}g\|_{L^{p'}(X, Y^*)}$. We fix $k_0$ to be so large that we have
\begin{displaymath}
\|T\|/2 \le |\langle g, Tf\rangle| \le \|T\|/4 + |\langle E^{b_2}_{k_0}g, T(E^{b_1}_{k_0}f)\rangle|,
\end{displaymath}
that is, $\|T\| \le 4|\langle E^{b_2}_{k_0}g, T(E^{b_1}_{k_0}f)\rangle|$ with any grids $\mathcal{D}$ and $\mathcal{D}'$ (but only with these particular fixed functions $f$ and $g$, of course).

Now we write $\langle E^{b_2}_{k_0}g, T(E^{b_1}_{k_0}f)\rangle$ as the following sum
\begin{align*}
&\Big \langle \mathop{\sum_{R \in \mathcal{D}_{\textrm{bad}}'}}_{\delta^{k_0} < \ell(R) \le \delta^m} \Delta^{b_2}_Rg + \mathop{\sum_{R \in \mathcal{D}_{\textrm{bad}}'}}_{\ell(R) = \delta^m} E^{b_2}_Rg, T(E^{b_1}_{k_0}f)\Big\rangle \\
&+\Big \langle \mathop{\sum_{R \in \mathcal{D}_{\textrm{good}}'}}_{\delta^{k_0} < \ell(R) \le \delta^m} \Delta^{b_2}_Rg + \mathop{\sum_{R \in \mathcal{D}_{\textrm{good}}'}}_{\ell(R) = \delta^m} E^{b_2}_Rg, 
T\Big(\mathop{\sum_{Q \in \mathcal{D}_{\textrm{bad}}}}_{\delta^{k_0} < \ell(Q) \le \delta^m} \Delta^{b_1}_Qf + \mathop{\sum_{Q \in \mathcal{D}_{\textrm{bad}}}}_{\ell(Q) = \delta^m} E^{b_1}_Qf\Big) \Big\rangle \\
&+\mathop{\sum_{Q \in \mathcal{D}_{\textrm{good}},\, R \in \mathcal{D}_{\textrm{good}}'}}_{\delta^{k_0} < \ell(Q), \, \ell(R) \le \delta^m} \sum_{u,v} \langle \varphi^{b_2}_{R,v}, g\rangle
\langle b_2\varphi^{b_2}_{R,v}, T(b_1\varphi^{b_1}_{Q,u})\rangle \langle \varphi^{b_1}_{Q,u} f\rangle,
\end{align*}
where the $u$ summation runs through $1, \ldots, s(Q)-1$ if $\ell(Q) < \delta^m$, and through $0, 1, \ldots, s(Q)-1$ if $\ell(Q) = \delta^m$, and similarly for the $v$ summation.
We thus have that $\|T\|/4$ is bounded by the sum of the following terms
\begin{displaymath}
\|T\| \Big\| \mathop{\sum_{R \in \mathcal{D}_{\textrm{bad}}'}}_{\delta^{k_0} < \ell(R) \le \delta^m} \Delta^{b_2}_Rg + \mathop{\sum_{R \in \mathcal{D}_{\textrm{bad}}'}}_{\ell(R) = \delta^m} E^{b_2}_Rg\Big \|_{L^{p'}(X,Y^*)}
\|E^{b_1}_{k_0}f\|_{L^p(X,Y)},
\end{displaymath}
\begin{displaymath}
\|T\| \Big\| \mathop{\sum_{R \in \mathcal{D}_{\textrm{good}}'}}_{\delta^{k_0} < \ell(R) \le \delta^m} \Delta^{b_2}_Rg + \mathop{\sum_{R \in \mathcal{D}_{\textrm{good}}'}}_{\ell(R) = \delta^m} E^{b_2}_Rg \Big \|_{L^{p'}(X,Y^*)}
\Big\|\mathop{\sum_{Q \in \mathcal{D}_{\textrm{bad}}}}_{\delta^{k_0} < \ell(Q) \le \delta^m} \Delta^{b_1}_Qf + \mathop{\sum_{Q \in \mathcal{D}_{\textrm{bad}}}}_{\ell(Q) = \delta^m} E^{b_1}_Qf \Big\|_{L^p(X,Y)}
\end{displaymath}
and
\begin{displaymath}
\Big|\mathop{\sum_{Q \in \mathcal{D}_{\textrm{good}},\, R \in \mathcal{D}_{\textrm{good}}'}}_{\delta^{k_0} < \ell(Q), \, \ell(R) \le \delta^m} \sum_{u,v} \langle \varphi^{b_2}_{R,v}, g\rangle
\langle b_2\varphi^{b_2}_{R,v}, T(b_1\varphi^{b_1}_{Q,u})\rangle \langle \varphi^{b_1}_{Q,u} f\rangle\Big|.
\end{displaymath}

Note that clearly $\|E^{b_1}_{k_0}f\|_{L^p(X,Y)} \lesssim \|f\|_{L^p(X,Y)} = 1$ and
\begin{displaymath}
\Big\| \mathop{\sum_{R \in \mathcal{D}_{\textrm{good}}'}}_{\ell(R) = \delta^m} E^{b_2}_Rg \Big\|_{L^{p'}(X,Y^*)} \lesssim \|g\|_{L^{p'}(X,Y^*)} = 1.
\end{displaymath}
Also, using unconditionality and the contraction principle, we have that
\begin{displaymath}
\Big\| \mathop{\sum_{R \in \mathcal{D}_{\textrm{good}}'}}_{\delta^{k_0} < \ell(R) \le \delta^m} \Delta^{b_2}_Rg \Big\|_{L^{p'}(X,Y^*)} \lesssim \|g\|_{L^{p'}(X,Y^*)} = 1.
\end{displaymath}

Thus, the terms involving bad cubes are dominated by
\begin{align*}
\|T\| \Bigg[ &\Big\| \mathop{\sum_{R \in \mathcal{D}_{\textrm{bad}}'}}_{\delta^{k_0} < \ell(R) \le \delta^m} \Delta^{b_2}_Rg + \mathop{\sum_{R \in \mathcal{D}_{\textrm{bad}}'}}_{\ell(R) = \delta^m} E^{b_2}_Rg\Big \|_{L^{p'}(X,Y^*)}\\
&+ \Big\|\mathop{\sum_{Q \in \mathcal{D}_{\textrm{bad}}}}_{\delta^{k_0} < \ell(Q) \le \delta^m} \Delta^{b_1}_Qf + \mathop{\sum_{Q \in \mathcal{D}_{\textrm{bad}}}}_{\ell(Q) = \delta^m} E^{b_1}_Qf \Big\|_{L^p(X,Y)}\Bigg].
\end{align*}
Taking expectations over all the random quantities in the randomization of cubes, it is easy to see that
\begin{displaymath}
\mathbb{E} \Big\|  \mathop{\sum_{R \in \mathcal{D}_{\textrm{bad}}'}}_{\ell(R) = \delta^m} E^{b_2}_Rg\Big \|_{L^{p'}(X,Y^*)} + \mathbb{E} \Big\| \mathop{\sum_{Q \in \mathcal{D}_{\textrm{bad}}}}_{\ell(Q) = \delta^m} E^{b_1}_Qf \Big\|_{L^p(X,Y)}
\lesssim \eta(r),
\end{displaymath}
where $\eta(r) \to 0$ when $r \to \infty$.
Working similarly as later in chapter 9 (when estimating a certain term $E_1$) we have that
\begin{displaymath}
\mathbb{E} \Big\| \mathop{\sum_{R \in \mathcal{D}_{\textrm{bad}}'}}_{\delta^{k_0} < \ell(R) \le \delta^m} \Delta^{b_2}_Rg \Big\|_{L^{p'}(X,Y^*)} +
\mathbb{E} \Big\|\mathop{\sum_{Q \in \mathcal{D}_{\textrm{bad}}}}_{\delta^{k_0} < \ell(Q) \le \delta^m} \Delta^{b_1}_Qf \Big\|_{L^p(X,Y)} \lesssim \eta(r)
\end{displaymath}
as well. The proof requires a certain improvement of the contraction principle which will also be recalled in chapter 9. One can consult \cite[chapter 12]{Hy2} too.

Choosing $r$ large enough we thus have that
\begin{equation}\label{normest}
\|T\|/8 \le \mathbb{E} \Big|\mathop{\sum_{Q \in \mathcal{D}_{\textrm{good}},\, R \in \mathcal{D}_{\textrm{good}}'}}_{\delta^{k_0} < \ell(Q), \, \ell(R) \le \delta^m} \sum_{u,v} \langle \varphi^{b_2}_{R,v}, g\rangle
\langle b_2\varphi^{b_2}_{R,v}, T(b_1\varphi^{b_1}_{Q,u})\rangle \langle \varphi^{b_1}_{Q,u} f\rangle\Big|.
\end{equation}
We almost always suppress the finite summation over $u,v$ and after that is done, simply write $\varphi_Q = \varphi^{b_1}_{Q,u}$, $\psi_R = \varphi^{b_2}_{R,v}$ and
$T_{RQ} = \langle b_2\psi_R, T(b_1\varphi_Q) \rangle$. The summation condition $\delta^{k_0} < \ell(Q), \, \ell(R) \le \delta^m$ is always in force, and thus most of the time not explicitly written.
The estimation of this series involving good cubes only is now split into multiple subseries to be considered in the subsequent chapters. We primarily deal with the part $\ell(Q) \le \ell(R)$ the other being symmetric.
Although we have $\|f\|_{L^p(X,Y)} = \|g\|_{L^{p'}(X,Y^*)} = 1$, in some of the estimates below we explicitly write $\|f\|_{L^p(X,Y)}$ and $\|g\|_{L^{p'}(X,Y^*)}$ in place of $1$ for clarity.

We still comment on some of the techniques used on the following chapters. Related to this vector-valued $L^p$-theory we combine basic randomization tricks with the more sophisticated tool
called the tangent martingale trick in \cite{Hy2}. Let us now formulate this since it is of fundamental importance to us.
\begin{prop}
Let $\mathcal{A} = \bigcup_k \mathcal{A}_k$, where $\mathcal{A}_k$ is a countable partition of $X$ into Borel sets of finite $\mu$-measure, and $\sigma(\mathcal{A}_k) \subset \sigma(\mathcal{A}_{k+1})$.
For each $A \in \mathcal{A}$ we are given a function $f_A\colon X \to Y$ supported on $A$, and
so that $f_A$ is $\sigma(\mathcal{A}_{k+1})$-measurable whenever $A \in \mathcal{A}_k$.
For each $A \in \mathcal{A}$ we are also given a jointly measurable function $k_A\colon A \times A \to \C$, which is pointwise bounded by $1$. We have
\begin{align*}
\int_{\Omega \times X}  \Big\| \sum_{k \in \Z}& \epsilon_k \sum_{A \in \mathcal{A}_k} \frac{\chi_A(x)}{\mu(A)} \int_A k_A(x,z)f_A(z) \,d\mu(z)\Big\|_Y^p\, d\mathbb{P}(\epsilon) \,d\mu(x) \\
&\lesssim \int_{\Omega \times X} \Big\| \sum_{k \in \Z} \epsilon_k \sum_{A \in \mathcal{A}_k} f_A(x)\Big\|_Y^p \,d\mathbb{P}(\epsilon)\,d\mu(x).
\end{align*}
\end{prop}
This is the only version of the trick we explicitly need in this paper. For this result and some more general theory related to this see \cite[chapter 6]{Hy2}. Lastly, we record the following randomization trick
which is used multiple times in the sequel. For the proof see \cite[p. 10]{Hy2}.
\begin{lem}\label{rantrick}
Suppose that for each $R \in \mathcal{D}'$ we are given a subcollection $\mathcal{D}(R) \subset \mathcal{D}$. There holds
\begin{align*}
\Big| &\sum_{R \in \mathcal{D}'} \langle g, \psi_R \rangle \sum_{Q \in \mathcal{D}(R)} T_{RQ} \langle \varphi_Q, f\rangle \Big| \\
&\lesssim \|g\|_{L^{p'}(X,Y^*)} \Big\| \sum_{k \in \Z} \epsilon_k \sum_{R \in \mathcal{D}'_k} \psi_R \sum_{Q \in \mathcal{D}(R)} T_{RQ} \langle \varphi_Q, f\rangle\Big\|_{L^p(\Omega \times X, Y)},
\end{align*}
where we have the measure $\mathbb{P} \times \mu$ on $\Omega \times X$ (here $(\Omega, \mathbb{P})$ is just some probability space).
\end{lem}
\section{Separated cubes}
We consider the part of the series where $R \in \mathcal{D}_{\textrm{good}}'$, $Q \in \mathcal{D}_{\textrm{good}}$, $\ell(Q) \le \ell(R)$ and $d(Q,R) \ge CC_0\ell(Q)$. Also the adapted Haar functions $\varphi_Q$
related to the smaller cubes $Q$ are assumed to be cancellative.

We begin with some estimates for the matrix elements $T_{RQ} = \langle b_2\psi_R, T(b_1\varphi_Q)\rangle$ -- these follow, with some modifications, \cite[Lemma 6.1 and Lemma 6.2]{HM}.
\begin{lem}
Let $Q \in \mathcal{D}$ and $R \in \mathcal{D}'$ be such that $\ell(Q) \le \ell(R)$ and $d(Q,R) \ge CC_0\ell(Q)$. Assume also that $\varphi_Q$ is cancellative.
We have the estimate
\begin{displaymath}
|T_{RQ}| \lesssim \frac{\ell(Q)^{\alpha}}{d(Q,R)^{\alpha}\sup_{z \in Q} \lambda(z, d(Q,R))}\|\varphi_Q\|_{L^1(\mu)}\|\psi_R\|_{L^1(\mu)}.
\end{displaymath}
\end{lem}
\begin{proof}
Recalling that $\int b_1\varphi_Q \,d\mu = 0$, we have for an arbitrary $z \in Q$ that
\begin{displaymath}
T_{RQ} = \int_R \int_Q [K(x,y)-K(x,z)]b_1(y)\varphi_Q(y)b_2(x)\psi_R(x)\,d\mu(y)\,d\mu(x).
\end{displaymath}
The claim follows from the kernel estimates (which we may utilize since $d(x,z) \ge d(Q,R) \ge CC_0\ell(Q) \ge Cd(y,z)$).
\end{proof}
We set $D(Q,R) = \ell(Q) + \ell(R) + d(Q,R)$.
\begin{lem}\label{seplem}
Let $Q \in \mathcal{D}_{\textrm{good}}$ and $R \in \mathcal{D}'$ be such that $\ell(Q) \le \ell(R)$ and $d(Q,R) \ge CC_0\ell(Q)$. Assume also that $\varphi_Q$ is cancellative.
We have the estimate
\begin{displaymath}
|T_{RQ}| \lesssim \frac{\ell(Q)^{\alpha/2}\ell(R)^{\alpha/2}}{D(Q,R)^{\alpha}\sup_{z \in Q} \lambda(z, D(Q,R))}\|\varphi_Q\|_{L^1(\mu)}\|\psi_R\|_{L^1(\mu)}.
\end{displaymath}
\end{lem}
\begin{proof}
If $\ell(Q) > \delta^r\ell(R)$, then $d(Q,R) \gtrsim D(Q,R)$, and the claim follows from the previous lemma. In the case $d(Q,R) \ge \ell(R)$, we also have $d(Q,R) \gtrsim D(Q,R)$, and the claim again follows from
the previous lemma.

We may thus assume that $\ell(Q) \le \delta^r\ell(R)$ and $d(Q,R) \le \ell(R)$. As $Q$ is good, we have
$d(Q,R) \gtrsim \ell(Q)^{\gamma}\ell(R)^{1-\gamma}$. Consider an arbitrary $z \in Q$. Using the identity
\begin{displaymath}
C_{\lambda}^{-\gamma \log_2 \frac{\ell(R)}{\ell(Q)}} = \Big(\frac{\ell(R)}{\ell(Q)}\Big)^{-\gamma d}
\end{displaymath}
and the doubling property of $\lambda$ one gets that
\begin{displaymath}
\lambda(z, d(Q,R)) \gtrsim \Big(\frac{\ell(R)}{\ell(Q)}\Big)^{-\gamma d} \lambda(z, \ell(R)).
\end{displaymath}
The claim then follows from the previous lemma, the identity $\gamma d + \gamma \alpha = \alpha/2$, and the fact that in our situation $\ell(R) \gtrsim D(Q,R)$.
\end{proof}
Let us then state and prove the main result of this section -- this follows, save the technical modifications, \cite[p. 25-26]{Hy2}. 
\begin{prop}\label{sepcubesprop}
There holds
\begin{displaymath}
\Big| \sum_{R \in \mathcal{D}'_{\textrm{good}}} \mathop{\sum_{Q \in \mathcal{D}_{\textrm{good}}}}_{\ell(Q) \le \ell(R), \, d(Q,R) \ge CC_0\ell(Q)} \langle g, \psi_R\rangle T_{RQ} \langle \varphi_Q, f\rangle\Big| \lesssim
\|g\|_{L^{p'}(X, Y^*)}\|f\|_{L^p(X, Y)}
\end{displaymath}
with the additional interpretation that the adapted Haar functions $\varphi_Q$ related to the smaller cubes $Q$ are cancellative, even on the coarsest level $\ell(Q) = \delta^m$.
\end{prop}
\begin{proof}
We first consider the case
\begin{displaymath}
\left\{ \begin{array}{ll}
\ell(R) = \delta^k, & k \in \Z, \\
\ell(Q) = \delta^{k+m}, & m = 0, 1, 2, \ldots, \\
\delta^{k-j} < D(Q,R) \le \delta^{k-j-1}, & j= 0, 1, 2, \ldots.
\end{array} \right.
\end{displaymath}
The last requirement says that $D(Q,R)/\ell(R) \sim \delta^{-j}$. The estimate from the previous lemma gives
\begin{displaymath}
\frac{|T_{RQ}|}{\|\varphi_Q\|_{L^1(\mu)}\|\psi_R\|_{L^1(\mu)}} \lesssim \frac{\delta^{\alpha m/2}\delta^{\alpha j}}{\sup_{z \in Q} \lambda(z, \delta^{k-j})}.
\end{displaymath}
We suppress from our notation the requirement that $d(Q,R) \ge CC_0\ell(Q)$.
Lemma \ref{rantrick} gives
\begin{align*}
\Big| \sum_{k \in \Z} \sum_{R \in \mathcal{D}_{\textrm{good}, k}} &\mathop{\sum_{Q \in \mathcal{D}_{\textrm{good}, k+m}}}_{D(Q,R)/\ell(R) \sim \delta^{-j}} \langle g, \psi_R\rangle T_{RQ} \langle \varphi_Q, f\rangle\Big|  \\
&\lesssim \|g\|_{L^{p'}(X, Y^*)}\Big\| \sum_{k \in \Z} \epsilon_k \sum_{R \in \mathcal{D}_{\textrm{good}, k}} \mathop{\sum_{Q \in \mathcal{D}_{\textrm{good}, k+m}}}_{D(Q,R)/\ell(R) \sim \delta^{-j}} \psi_R T_{RQ} \langle \varphi_Q, f\rangle
\Big\|_{L^p(\Omega \times X, Y)}.
\end{align*}

For a cube $Q$ denote by $\tilde Q_{\ell}$ the unique cube of generation $\ell \le \textrm{gen}(Q)$ for which $Q \subset \tilde Q_{\ell}$. Let $\theta(j)$ denote the smallest integer for which
$\theta(j) \ge (j\gamma + r)(1-\gamma)^{-1}$. Recalling that $R$ is good and $r$ is large enough, we must have for any $Q$ and $R$ in the above summation that
$R \subset \tilde Q_{k-j-\theta(j)}$. Thus, we may write
\begin{displaymath}
\sum_{R \in \mathcal{D}'_{\textrm{good},k}} = \sum_{S \in \mathcal{D}_{k-j-\theta(j)}} \mathop{\sum_{R \in \mathcal{D}'_{\textrm{good},k}}}_{R \subset S}.
\end{displaymath}
Also, we have $\mu(S) \lesssim \inf_{w \in S} \lambda(w, \delta^{k-j-\theta(j)}) \lesssim \delta^{-d\theta(j)} \inf_{w \in S}  \lambda(w, \delta^{k-j})$. Define $t_{RQ}$ via the identity
\begin{displaymath}
T_{RQ} = \frac{\delta^{\alpha m/2}\delta^{\alpha j - d\theta(j)}}{\mu(S)} \|\varphi_Q\|_{L^1(\mu)}\|\psi_R\|_{L^1(\mu)}t_{RQ},
\end{displaymath}
and note that we have
\begin{displaymath}
|t_{RQ}| \lesssim \frac{\inf_{w \in S} \lambda(w, \delta^{k-j})}{\sup_{z \in Q} \lambda(z, \delta^{k-j})} \le 1.
\end{displaymath}
Also relevant is the estimate
\begin{displaymath}
\delta^{\alpha j - d\theta(j)} \lesssim \delta^{[\alpha-d\gamma(1-\gamma)^{-1}]j} = \delta^{(\alpha^2+\alpha d)(\alpha+2d)^{-1}j}.
\end{displaymath}

For every $S \in \mathcal{D}_{k-j-\theta(j)}$ we set
\begin{displaymath}
K_S(x,y) = \mathop{\sum_{R \in \mathcal{D}'_{\textrm{good},k}}}_{R \subset S} \mathop{\sum_{Q \in \mathcal{D}_{\textrm{good}, k+m}}}_{D(Q,R)/\ell(R) \sim \delta^{-j}} \psi_R(x)\|\psi_R\|_{L^1(\mu)}t_{RQ}\|\varphi_Q\|_{L^1(\mu)} \varphi_Q(y)b_1(y).
\end{displaymath}
As $\|\varphi_Q\|_{L^1(\mu)}\|\varphi_Q\|_{L^{\infty}(\mu)} \lesssim 1$, $\|\psi_R\|_{L^1(\mu)}\|\psi_R\|_{L^{\infty}(\mu)} \lesssim 1$, $\|b_1\|_{L^{\infty}(\mu)} \lesssim 1$, $|t_{RQ}| \lesssim 1$ and
for every fixed $x$ and $y$ there is at most one non-zero term in the double sum defining $K_S$, we have $|K_S(x,y)| \lesssim 1$. Also, $K_S$ is supported on $S \times S$ as spt$\,\psi_R \subset R \subset S$
and spt$\,\varphi_Q \subset Q \subset S$.

Using the fact that $\int b_1\varphi_Q\,d\mu = 0$ one notes that $\langle \varphi_Q, f\rangle = \langle \varphi_Q, \Delta^{b_1}_{k+m}f\rangle$ for $Q \in \mathcal{D}_{k+m}$. Using this and the definitions from above, we see that
\begin{align*}
\Big\| &\sum_{k \in \Z} \epsilon_k \sum_{R \in \mathcal{D}_{\textrm{good}, k}} \mathop{\sum_{Q \in \mathcal{D}_{\textrm{good}, k+m}}}_{D(Q,R)/\ell(R) \sim \delta^{-j}} \psi_R T_{RQ} \langle \varphi_Q, f\rangle\Big\|_{L^p(\Omega \times X, Y)} \\
&\lesssim \delta^{\frac{\alpha}{2}m}\delta^{\frac{\alpha^2+\alpha d}{\alpha+2d}j} \Big\| \sum_{k \in \Z} \epsilon_k \sum_{S \in \mathcal{D}_{k-j-\theta(j)}} \frac{\chi_S}{\mu(S)} \int_S K_S(\cdot, y)\frac{\chi_S(y)\Delta^{b_1}_{k+m}f(y)}{b_1(y)}\,d\mu(y)
\Big\|_{L^p(\Omega \times X, Y)}.
\end{align*}
Due to the measurability requirements of the tangent martingale trick we further split up the above sum over $k \in \Z$ into $m+j+\theta(j)+1 \lesssim m+j+1$ subseries:
\begin{displaymath}
\sum_{k \in \Z} = \sum_{k_0=0}^{m+j+\theta(j)} \mathop{\sum_{k \equiv k_0}}_{\textrm{mod}\,m+j+\theta(j)+1}.
\end{displaymath}
The point is that $y \mapsto \frac{\chi_S(y)\Delta^{b_1}_{k+m}f(y)}{b_1(y)}$ is constant on the subcubes of generation $k+m+1 = k' - j - \theta(j)$, where $k' = k + (m + j + \theta(j) +1)$.
Applying the tangent martingale trick to each of these subseries then yields that
\begin{align*}
\Big| &\sum_{k \in \Z} \sum_{R \in \mathcal{D}_{\textrm{good}, k}} \mathop{\sum_{Q \in \mathcal{D}_{\textrm{good}, k+m}}}_{D(Q,R)/\ell(R) \sim \delta^{-j}} \langle g, \psi_R\rangle T_{RQ} \langle \varphi_Q, f\rangle\Big| \\
&\lesssim \delta^{\frac{\alpha}{2}m}\delta^{\frac{\alpha^2+\alpha d}{\alpha+2d}j} \|g\|_{L^{p'}(X, Y^*)} \sum_{k_0=0}^{m+j+\theta(j)} 
\Big\| \sum_{k \equiv k_0} \epsilon_k \sum_{S \in \mathcal{D}_{k-j-\theta(j)}} \frac{\chi_S\Delta^{b_1}_{k+m}f}{b_1}\Big\|_{L^p(\Omega \times X, Y)} \\
&\lesssim \delta^{\frac{\alpha}{2}m}\delta^{\frac{\alpha^2+\alpha d}{\alpha+2d}j}(m+j+1) \|g\|_{L^{p'}(X, Y^*)} \|f\|_{L^p(X,Y)},
\end{align*}
where the last inequality follows from the unconditional convergence of the adapted martingale difference decomposition (after discarding $1/b_1$). Summing over $m,j = 0, 1, 2, \ldots$ yields the claim.
\end{proof}
\section{Cubes well inside another cube}
We consider the case $R \in \mathcal{D}_{\textrm{good}}'$, $Q \in \mathcal{D}_{\textrm{good}}$, $Q \subset R$ and $\ell(Q) < \delta^r\ell(R)$. As usual, there is a need to introduce some cancellation. To this end, here we consider
the modified matrix
\begin{align*}
\tilde T_{RQ} &= T_{RQ} - \langle b_2, T(b_1\varphi_Q)\rangle\langle \psi_R\rangle_Q\\
&= -\langle \chi_{X \setminus S}b_2, T(b_1\varphi_Q)\rangle \langle \psi_R \rangle_S + \mathop{\sum_{S' \subset R \setminus S}}_{\ell(S') = \delta\ell(R)}
\langle \chi_{S'}\psi_Rb_2, T(b_1\varphi_Q)\rangle, 
\end{align*}
where $S \subset R$ is such that $\ell(S) = \delta\ell(R)$ and $Q \subset S$. The point is that $Q$ is separated from the rest of the subcubes $S'$ and we have introduced cancellation for this one problematic subcube $S$. The correction terms
form a paraproduct operator, the boundedness of which will be considered in the next chapter.

We again begin with some estimates for the matrix $\tilde T_{RQ}$. Let us be brief as these estimates follow pretty much as in \cite[p. 20-21]{HM}. Fix some $z \in Q$. Recalling that for every ball $B = B(c_B, r_B)$ and for every $\epsilon > 0$
we have the estimate (integrate over dyadic blocks $2^jr_B \le d(x, c_B) < 2^{j+1}r_B$ or see \cite[Lemma 2.4]{HM})
\begin{displaymath}
\int_{X \setminus B} \frac{d(x,c_B)^{-\epsilon}}{\lambda(c_B, d(x,c_B))}\,d\mu(x) \lesssim_{\epsilon} r_B^{-\epsilon},
\end{displaymath}
we establish by changing $K(x,y)$ to $K(x,y)-K(x,z)$ (using $\int b_1\varphi_Q\,d\mu = 0$), using the kernel estimates and noting that $X \setminus S \subset X \setminus B(z, d(Q, X \setminus S))$ that
\begin{displaymath}
|\langle \chi_{X \setminus S}b_2, T(b_1\varphi_Q)\rangle| \lesssim \ell(Q)^{\alpha}\|\varphi_Q\|_{L^1(\mu)}d(Q, X \setminus S)^{-\alpha}.
\end{displaymath}
To see that it was legitimate to use the kernel estimates note that in the corresponding integral $d(x,z) \ge d(X \setminus S, Q) \gtrsim \ell(Q)^{\gamma}\ell(S)^{1-\gamma} \ge \delta^{-r(1-\gamma)}\ell(Q)$, so that
$d(x,z) \ge Cd(y,z)$ choosing $r$ large enough. Furthermore,
note that $d(Q, X \setminus S) \gtrsim \ell(Q)^{\gamma}\ell(S)^{1-\gamma} \ge \ell(Q)^{1/2}\ell(R)^{1/2}$, and so continuing the above estimates we obtain
\begin{displaymath}
|\langle \chi_{X \setminus S}b_2, T(b_1\varphi_Q)\rangle| \lesssim \Big(\frac{\ell(Q)}{\ell(R)}\Big)^{\alpha/2}\|\varphi_Q\|_{L^1(\mu)}.
\end{displaymath}

For the other finitely many terms involving a subcube $S' \subset R$ (where we have separation) we have using Lemma \ref{seplem} (or actually, a trivial modification) that
\begin{align*} 
|\langle \chi_{S'}\psi_Rb_2, T(b_1\varphi_Q)\rangle| &\lesssim \Big(\frac{\ell(Q)}{\ell(S')}\Big)^{\alpha/2} \frac{\|\psi_R\|_{L^1(\mu)}}{\lambda(z, \ell(S'))}\|\varphi_Q\|_{L^1(\mu)}\\
&\lesssim \Big(\frac{\ell(Q)}{\ell(R)}\Big)^{\alpha/2} \frac{\|\psi_R\|_{L^1(\mu)}}{\mu(R)}\|\varphi_Q\|_{L^1(\mu)},
\end{align*}
where the last estimate follows after noting that
\begin{displaymath}
\mu(R) \le \mu(B(z,C_0\ell(R))) \le \lambda(z, C_0\ell(R)) = \lambda(z, C_0\delta^{-1}\ell(S')) \lesssim \lambda(z, \ell(S')).
\end{displaymath}
Let us recapitulate all this as a lemma.
\begin{lem}
If $R \in \mathcal{D}'$, $Q \in \mathcal{D}_{\textrm{good}}$, $Q \subset R$, $\ell(Q) < \delta^r\ell(R)$ and $S$ is the subcube of $R$ for which $\ell(S) = \delta\ell(R)$ and $Q \subset S$, we have
\begin{displaymath}
|\tilde T_{RQ}| \lesssim \Big(\frac{\ell(Q)}{\ell(R)}\Big)^{\alpha/2}\Big[|\langle \psi_R\rangle_S| + \frac{\|\psi_R\|_{L^1(\mu)}}{\mu(R)}\Big]\|\varphi_Q\|_{L^1(\mu)}.
\end{displaymath}
\end{lem}
A familiar strategy involving kernels and the tangent martingale trick shall now be employed (as in the previous chapter and as in \cite{Hy2}). For this, the following lemma is both natural and useful.
\begin{lem}
If $R \in \mathcal{D}'$, $Q \in \mathcal{D}_{\textrm{good}}$, $Q \subset R$, $\ell(Q) < \delta^r\ell(R)$ and $S$ is the subcube of $R$ for which $\ell(S) = \delta\ell(R)$ and $Q \subset S$, we have
\begin{displaymath}
|\psi_R(x)\tilde T_{RQ}\varphi_Q(y)| \lesssim \Big(\frac{\ell(Q)}{\ell(R)}\Big)^{\alpha/2}\Big[\frac{\chi_{R \setminus S}(x)}{\mu(R)} + \frac{\chi_S(x)}{\mu(S)}\Big].
\end{displaymath}
\end{lem}
\begin{proof}
Taking the previous lemma and the estimates $\|\varphi_Q\|_{L^1(\mu)}\|\varphi_Q\|_{L^{\infty}(\mu)} \lesssim 1$ and $\|\psi_R\|_{L^1(\mu)}\|\psi_R\|_{L^{\infty}(\mu)} \lesssim 1$ into account it suffices to prove that
\begin{displaymath}
|\langle \psi_R\rangle_S||\psi_R(x)| \lesssim \frac{\chi_{R \setminus S}(x)}{\mu(R)} + \frac{\chi_S(x)}{\mu(S)}.
\end{displaymath}
This follows by recalling that $\psi_R = \varphi_{R, v}^{b_2}$ for some $v$, denoting $S = R_w$, subdividing the estimation into cases ($v=w$ and $x \in S$), ($v=w$ and $x \in R \setminus S$) and $v \ne w$, and finally recalling
that one has
\begin{displaymath}
|\psi_R| \sim \mu(R_v)^{1/2}\Big(\frac{\chi_{R_v}}{\mu(R_v)} + \frac{\chi_{\hat R_{v+1}}}{\mu(R)}\Big)
\end{displaymath}
(or $|\psi_R| \sim \mu(R)^{-1/2}$ if $v=0$ and no subdivision into cases is necessary).
\end{proof}
We are now ready to prove the main result of this section. 
\begin{prop}
It holds
\begin{displaymath}
\Big| \sum_{R \in \mathcal{D}_{\textrm{good}}'} \mathop{\sum_{Q \in \mathcal{D}_{\textrm{good}}, \, Q \subset R}}_{\ell(Q) < \delta^r\ell(R)} \langle g, \psi_R \rangle \tilde T_{RQ} \langle \varphi_Q, f \rangle \Big| \lesssim 
\|g\|_{L^{p'}(X,Y^*)}\|f\|_{L^p(X, Y)}.
\end{displaymath}
\end{prop}
\begin{proof}
Let $s(R)$ denote the number of subcubes of a cube $R \in \mathcal{D}'$ and set $s = \max_{R \in \mathcal{D}'} s(R) \lesssim 1$. Fix $w \in \{1, \ldots, s\}$ and $m \in \{r+1, r+2, \ldots\}$.
The already used randomization trick gives
\begin{align*}
\Big| \sum_{k \in \Z} \sum_{R \in \mathcal{D}_{\textrm{good},k}'}& \mathop{\sum_{Q \in \mathcal{D}_{\textrm{good},k+m}}}_{Q \subset R_w} \langle g, \psi_R \rangle \tilde T_{RQ} \langle \varphi_Q, f \rangle \Big| \\
&\lesssim \|g\|_{L^{p'}(X,Y^*)} \Big\|\sum_{k \in \Z} \epsilon_k \sum_{R \in \mathcal{D}_{\textrm{good},k}'} \mathop{\sum_{Q \in \mathcal{D}_{\textrm{good},k+m}}}_{Q \subset R_w}
\psi_R \tilde T_{RQ} \langle \varphi_Q, f \rangle\Big\|_{L^p(\Omega \times X, Y)}.
\end{align*}

We introduce the relevant kernels now. Indeed, set
\begin{align*}
K_R^c &= \delta^{-\alpha m/2}  \mathop{\sum_{Q \in \mathcal{D}_{\textrm{good},k+m}}}_{Q \subset R_w} \mu(R)\chi_{R \setminus R_w}(x)\psi_R(x)\tilde T_{RQ} \varphi_Q(y)b_1(y), \\
K_R^i &= \delta^{-\alpha m/2}  \mathop{\sum_{Q \in \mathcal{D}_{\textrm{good},k+m}}}_{Q \subset R_w} \mu(R_w)\chi_{R_w}(x)\psi_R(x)\tilde T_{RQ} \varphi_Q(y)b_1(y).
\end{align*}
The previous lemma yields at once that $|K_R^c(x,y)| \lesssim 1$ and $|K_R^i(x,y)| \lesssim 1$. Also, the supports lie in $R \times R$ and $R_w \times R_w$ respectively. There holds
\begin{align*}
\sum_{k \in \Z} &\epsilon_k \sum_{R \in \mathcal{D}_{\textrm{good},k}'} \mathop{\sum_{Q \in \mathcal{D}_{\textrm{good},k+m}}}_{Q \subset R_w} \psi_R(x) \tilde T_{RQ} \langle \varphi_Q, f \rangle \\
&= \delta^{\alpha m/2}\sum_{k \in \Z} \epsilon_k \sum_{R \in \mathcal{D}_{\textrm{good},k}'}  \frac{\chi_R(x)}{\mu(R)}
\int_R K^c_R(x,y)\frac{\chi_R(y)\Delta^{b_1}_{k+m}f(y)}{b_1(y)}\,d\mu(y) \\
&+ \delta^{\alpha m/2} \sum_{k \in \Z} \epsilon_k \sum_{R \in \mathcal{D}_{\textrm{good},k}'} \frac{\chi_{R_w}(x)}{\mu(R_w)}
\int_{R_w} K^i_R(x,y)\frac{\chi_{R_w}(y)\Delta^{b_1}_{k+m}f(y)}{b_1(y)}\,d\mu(y).
\end{align*}

The tangent martingale trick cannot quite yet be used -- the measurability conditions need not hold (note the important difference with the argument of the previous section -- there we did not have
the dyadic systems $\mathcal{D}$ and $\mathcal{D}'$ mixed in the way we have here). To fix this, one simply defines new partitions
\begin{displaymath}
\mathcal{F}_k = \{S \cap Q \ne \emptyset: S \in \mathcal{D}'_k,\, Q \in \mathcal{D}_{k-r-1}\}
\end{displaymath}
and exploits the goodness of the cubes $R$ via the observations
\begin{displaymath}
\mathcal{D}'_{\textrm{good},k} \subset \mathcal{F}_k \qquad \textrm{and} \qquad \{R_w \in \mathcal{D}_{k+1}': R_w \subset R \in \mathcal{D}_{\textrm{good},k}'\} \subset \mathcal{F}_{k+1}. 
\end{displaymath}

We then extend the above sums to be over the sets $\mathcal{F}_k$ and $\mathcal{F}_{k+1}$ respectively by using zero kernels for all the new sets $R$. We may then apply the tangent martingale trick
after passing to the obvious subseries over $k$ yielding, just like in the previous section, the bound
\begin{displaymath}
\Big| \sum_{k \in \Z} \sum_{R \in \mathcal{D}_{\textrm{good},k}'} \mathop{\sum_{Q \in \mathcal{D}_{\textrm{good},k+m}}}_{Q \subset R_w} \langle g, \psi_R \rangle \tilde T_{RQ} \langle \varphi_Q, f \rangle \Big|
\lesssim \delta^{\alpha m/2} (m+r+1) \|g\|_{L^{p'}(X,Y^*)}\|f\|_{L^p(X,Y)},
\end{displaymath}
from which the claim follows after summing over $m = r+1, r+2, \ldots$ and $w = 1, \ldots, s$.
\end{proof}
\section{The correction term and the relevant paraproduct}
Recall that we subtracted $\langle b_2, T(b_1\varphi_Q)\rangle\langle \psi_R\rangle_Q$ from $T_{RQ}$ in the case
$R \in \mathcal{D}_{\textrm{good}}'$, $Q \in \mathcal{D}_{\textrm{good}}$, $Q \subset R$ and $\ell(Q) < \delta^r\ell(R)$.
Thus, we now need to consider the sum
\begin{equation}\label{par1}
\sum_{R \in \mathcal{D}_{\textrm{good}}'} \mathop{\sum_{Q \in \mathcal{D}_{\textrm{good}}, \, Q \subset R}}_{\ell(Q) < \delta^r\ell(R)} \langle g, \psi_R \rangle \langle b_2, T(b_1\varphi_Q)\rangle\langle \psi_R\rangle_Q \langle \varphi_Q, f \rangle.
\end{equation}
Recall also that we always have the suppressed summation over $u,v$ and the restriction that $\delta^{k_0} < \ell(Q), \, \ell(R) \le \delta^m$. Writing out the above sum unhiding these conventions and then recalling
that e.g. $\Delta^{b_1}_Qf = \sum_u b_1\varphi_{Q,u} \langle \varphi_{Q,u}, f\rangle$, we see that (writing explicitly only the relevant restrictions)
\begin{align*}
(\ref{par1}) = \mathop{\sum_{Q \in \mathcal{D}_{\textrm{good}}}}_{\ell(Q) > \delta^{k_0}} \Big( &\mathop{\sum_{R \in \mathcal{D}_{\textrm{good}}', \, R \supset Q}}_{\delta^{-r}\ell(Q) < \ell(R) \le \delta^m} \langle \Delta^{b_2}_Rg / b_2\rangle_Q \\
&+ \mathop{\sum_{R \in \mathcal{D}_{\textrm{good}}', \, R \supset Q}}_{\delta^{-r}\ell(Q) < \ell(R) = \delta^m} \langle E^{b_2}_Rg / b_2\rangle_Q \Big)
\langle T^*b_2, \Delta^{b_1}_Qf\rangle.
\end{align*}

Now we use the trick from \cite{Hy2} noting that the inner summation would collapse to $\langle E^{b_2}_R g/b_2\rangle_Q = \langle g\rangle_R / \langle b_2 \rangle_R$, where
$R \in \mathcal{D}'$ is the unique cube of generation gen$(Q)-r$ for which $Q \subset R$, were it not for the restriction to good $\mathcal{D}'$-cubes in the summation. Now it is clear
why Lemma \ref{goodtonogood} was worth proving. Indeed, we may achieve this effect just by considering the grid $\mathcal{D}'$ being fixed and averaging over all the other random quantities used in the randomization of cubes.
We use Lemma \ref{goodtonogood} twice. First, to remove the restriction to good $R$, and after collapsing the series, to put the restriction back. This yields
\begin{align*}
\mathbb{E}(\ref{par1}) &= \mathbb{E} \sum_{Q \in \mathcal{D}_{\textrm{good}}} \mathop{\sum_{R \in \mathcal{D}_{\textrm{good}}', \, R \supset Q}}_{\ell(R) = \delta^{-r}\ell(Q)} \frac{\langle g\rangle_R}
{\langle b_2 \rangle_R} \langle T^*b_2, \Delta^{b_1}_Qf\rangle \\
&= \mathbb{E} \sum_{R \in \mathcal{D}_{\textrm{good}}'} \mathop{\sum_{Q \in \mathcal{D}_{\textrm{good}}, \, Q \subset R}}_{\ell(Q) = \delta^r\ell(R)} \frac{\langle g\rangle_R}
{\langle b_2 \rangle_R} \langle T^*b_2, b_1\varphi_Q \rangle\langle \varphi_Q, f\rangle,
\end{align*}
where the standard summation conditions were yet again suppressed.

Notice now that the right hand side of this is the expectation of a pairing $\langle \Pi g, f\rangle$, where we have (for every fixed choice of the random quantities) the paraproduct
\begin{displaymath}
\Pi g = \sum_{R \in \mathcal{D}_{\textrm{good}}'} \mathop{\sum_{Q \in \mathcal{D}_{\textrm{good}}, \, Q \subset R}}_{\ell(Q) = \delta^r\ell(R)} \frac{\langle g\rangle_R}
{\langle b_2 \rangle_R} \langle T^*b_2, b_1\varphi_Q \rangle \varphi_Q.
\end{displaymath}
We shall next study this with any fixed choice of the random quantities. Note that in \cite{HM} the paraproduct had the inessential difference that instead of the requirement of $Q$ being good we had
the requirement $d(Q, X \setminus R) \ge CC_0\ell(Q)$ (which follows from the goodness), and the essential difference that the bigger cubes were not restricted to good cubes. As was noted in
\cite{Hy2}, this restriction is useful in this vector valued context.
\begin{lem}
If $\varphi \in \textrm{BMO}^p_{\kappa}(\mu)$, then
\begin{displaymath}
\Big\| \mathop{\sum_{Q \in \mathcal{D}_{\textrm{good}}, \, Q \subset R}}_{\ell(Q) \le \delta^r\ell(R)} \epsilon_Q \langle \varphi, b_1\varphi_Q \rangle \varphi_Q \Big\|_{L^p(\Omega \times X)}
\lesssim \mu(R)^{1/p}\|\varphi\|_{\textrm{BMO}^p_{\kappa}(\mu)}.
\end{displaymath}
\end{lem}
\begin{proof}
This can be proven similarly as \cite[Lemma 7.1]{HM} borrowing some minor additional ingredients related to this vector valued context from the proof of \cite[Lemma 9.3]{Hy2}. 
\end{proof}
Since $T^*b_2 \in \textrm{RBMO}(\mu) \subset \textrm{BMO}^p_{\kappa}(\mu)$ for any $1 \le p < \infty$ (see the relevant chapter of the present work), the previous lemma is important in proving that the paraproduct $\Pi$ is bounded.
We will not provide the exact details instead citing \cite{Hy2} as this part of the argument no longer has anything special to do with the metric space structure or with our use of more general measures.
Indeed, having been able to do all these reductions in the metric space setting, one can now follow the argument found in \cite[p. 32-33]{Hy2} pretty much word to word (when reading that, notice that the chapter 3 of \cite{Hy2} is already in a abstact form suitable for us), and this yields:
\begin{prop}
We have
\begin{displaymath}
\|\Pi g\|_{L^{p'}(X, Y^*)} \lesssim \|T^*b_2\|_{\textrm{BMO}^{p'}_{\kappa}(\mu)} \|g\|_{L^{p'}(X,Y^*)} \lesssim \|g\|_{L^{p'}(X,Y^*)}.
\end{displaymath}
\end{prop}
The main result of this chapter now readily follows.
\begin{prop}
We have
\begin{displaymath}
\Big|\mathbb{E} \sum_{R \in \mathcal{D}_{\textrm{good}}'} \mathop{\sum_{Q \in \mathcal{D}_{\textrm{good}}, \, Q \subset R}}_{\ell(Q) < \delta^r\ell(R)} \langle g, \psi_R \rangle \langle b_2, T(b_1\varphi_Q)\rangle\langle \psi_R\rangle_Q \langle \varphi_Q, f \rangle\Big| \le \|g\|_{L^{p'}(X,Y^*)}\|f\|_{L^p(X,Y)},
\end{displaymath}
where we average over all the random quantities used in the randomization of the cubes.
\end{prop}
\section{Estimates for adjacent cubes of comparable size}
We shall now deal with the part of the series where good cubes $Q \in \mathcal{D}_{\textrm{good}}$ and $R \in \mathcal{D}_{\textrm{good}}'$ are adjacent ($d(Q, R) < CC_0\min(\ell(Q), \ell(R))$) and of
comparable size ($|\textrm{gen}(Q)-\textrm{gen}(R)| \le r$). We denote the last condition by $\ell(Q) \sim \ell(R)$.
Also, only the size, and not the cancellation, properties of the adapted Haar functions are used.

We are given some fixed small $\epsilon > 0$. Given cubes $Q$ and $R$ define $\Delta = Q \cap R$, $\delta_Q = \{x: d(x,Q) \le \epsilon \ell(Q) \textrm{ and } d(x, X \setminus Q) \le \epsilon \ell(Q)\}$
and $\delta_R = \{x: d(x, R) \le \epsilon \ell(R) \textrm{ and } d(x, X \setminus R) \le \epsilon \ell(R)\}$. Also, set
\begin{displaymath}
Q_b = Q \cap \bigcup_{R' \in \mathcal{D}':\, \ell(R') \sim \ell(Q)} \delta_{R'}
\end{displaymath}
and
\begin{displaymath}
R_b = R \cap \bigcup_{Q' \in \mathcal{D}:\, \ell(Q') \sim \ell(R)} \delta_{Q'}.
\end{displaymath}
Set also $Q_s = Q \setminus \Delta \setminus \delta_R$, $Q_{\partial} = Q \setminus \Delta \setminus Q_s$,
$R_s = R \setminus \Delta \setminus \delta_Q$ and $R_{\partial} = R \setminus \Delta \setminus R_s$. Furthermore, we still define that
$\tilde \Delta = \Delta \setminus \delta_Q \setminus \delta_R$.

Given $R \in \mathcal{D}_{\textrm{good}}'$, there are only finitely many $Q \in  \mathcal{D}_{\textrm{good}}$ which are adjacent to $R$ and of comparable size. Thus, one needs only to study finitely many subseries
\begin{displaymath}
\sum_{R \in \mathcal{D}_{\textrm{good}}'} \langle g, \psi_R\rangle T_{RQ} \langle \varphi_Q, f\rangle,
\end{displaymath}
where $Q = Q(R)$ is implicitly a function of $R$ -- a convention that is used throughout this section. We shall also act like the mapping $R \mapsto Q(R)$ is invertible -- this only amounts to
identifying some terms with zero (if there are no preimages) or splitting into finitely many new subseries using the triangle inequality (if there are multiple preimages).

Recall that $T_{RQ} = \langle \psi_Rb_2, T(b_1\varphi_Q)\rangle$. We note that
\begin{displaymath}
b_1\varphi_Q \langle \varphi_Q, f\rangle = \mathop{\sum_{Q' \in \mathcal{D}: Q' \subset Q}}_{\ell(Q') = \delta\ell(Q)} b_1\chi_{Q'}\langle \varphi_Q \rangle_{Q'} \langle \varphi_Q, f\rangle
= \mathop{\sum_{Q' \in \mathcal{D}: Q' \subset Q}}_{\ell(Q') = \delta\ell(Q)} b_1\chi_{Q'}A_{Q'},
\end{displaymath}
where $A_{Q'} = \langle \varphi_Q \rangle_{Q'} \langle \varphi_Q, f\rangle$.
Similarly there holds
\begin{displaymath}
b_2\psi_R \langle g,\psi_R\rangle = \mathop{\sum_{R' \in \mathcal{D}': R' \subset R}}_{\ell(R') = \delta\ell(R)} b_2\chi_{R'}B_{R'},
\end{displaymath}
where $B_{R'} = \langle \psi_R \rangle_{R'} \langle g,\psi_R\rangle$. Thus, we are left with finitely many new subseries
of the form
\begin{displaymath}
\sum_{R \in \mathcal{D}'} B_R\langle \chi_R b_2, T(b_1\chi_Q)\rangle A_Q,
\end{displaymath}
where $Q = Q(R)$ is a new function of $R$ but one still has $\ell(Q) \sim \ell(R)$. Note also that the parents of these cubes are always good.

Given $R$ and then $Q = Q(R)$ as in the above sum, we shall now split the pairing $\langle \chi_R b_2, T(b_1\chi_Q)\rangle$ into several terms. First, we use that given $\upsilon \in (0,1)$ there exists an
almost-covering $\mathcal{B}$ of $\tilde \Delta$ by separated balls in the sense that we have the following properties:
\begin{displaymath}
\left\{ \begin{array}{l}
\mu(\tilde \Delta \setminus \bigcup_{B \in \mathcal{B}} B) \le \upsilon\mu(\tilde \Delta), \\
\Lambda B \subset \Delta \textrm{ for every } B \in \mathcal{B}, \\
d(B, B') \gtrsim_{\upsilon} \max(r_B, r_{B'}) \textrm{ if } B, B' \in \mathcal{B}, \,\, B \ne B', \\
\#\mathcal{B} \lesssim C(\epsilon, \upsilon).
\end{array} \right.
\end{displaymath}
For the details of the probabilistic construction of $\mathcal{B}$, see chapters 8 and 9 of \cite{HM}. We write $\Delta \setminus \bigcup B$ as a disjoint union of $\Omega_i = \tilde \Delta \setminus \bigcup B$ and some
sets $\Omega_Q \subset Q_b$ and $\Omega_R \subset R_b$.

We now decompose
\begin{eqnarray*}
\langle \chi_R b_2,T(b_1\chi_Q)\rangle & = & \langle \chi_{R_{\partial}} b_2, T(b_1\chi_Q) \rangle + \langle \chi_{R_s} b_2, T(b_1\chi_Q) \rangle \\
& + & \langle \chi_{\Delta} b_2,T(b_1\chi_{Q_{\partial}}) \rangle + \langle \chi_{\Delta} b_2,T(b_1\chi_{Q_s}) \rangle \\
& + & \langle \chi_{\Delta \setminus \bigcup B} b_2, T(b_1\chi_{\Delta})\rangle + \langle \chi_{\bigcup B} b_2, T(b_1\chi_{\Delta \setminus \bigcup B}) \rangle \\
& + & \langle \chi_{\bigcup B} b_2, T(b_1\chi_{\bigcup B})\rangle = A + B + C + D + E + F + G.
\end{eqnarray*}
Furhermore, we decompose
\begin{align*}
E &= \langle \chi_{\Delta \setminus \bigcup B} b_2, T(b_1\chi_{\Delta}) \rangle \\
   &= \langle \chi_{\Omega_Q} b_2, T(b_1\chi_{\Delta}) \rangle + \langle \chi_{\Omega_R} b_2, T(b_1\chi_{\Delta}) \rangle
+ \langle \chi_{\Omega_i} b_2, T(b_1\chi_{\Delta}) \rangle\\ &= E_1 + E_2 + E_3
\end{align*}
and
\begin{align*}
F &= \langle \chi_{\bigcup B} b_2, T(b_1\chi_{\Delta \setminus \bigcup B}) \rangle \\
   &= \langle \chi_{\bigcup B} b_2, T(b_1\chi_{\Omega_Q}) \rangle + \langle \chi_{\bigcup B} b_2, T(b_1\chi_{\Omega_R}) \rangle
    +\langle \chi_{\bigcup B} b_2, T(b_1\chi_{\Omega_i}) \rangle \\
    &= F_1 + F_2 + F_3.
\end{align*}
We still write
\begin{align*}
G &= \langle \chi_{\bigcup B} b_2, T(b_1\chi_{\bigcup B}) \rangle \\
   &= \sum_B \langle \chi_B b_2, T(b_1\chi_B) \rangle + \sum_{B \ne B'} \langle \chi_{B'} b_2, T(b_1\chi_B) \rangle = G_1 + G_2.
\end{align*}

It is time to deal with these terms now. These belong to various different groups: we have the terms with separation $B, D$ and $G_2$, the terms $C, E_1$ and $F_1$ involving the bad boundary region $Q_b$,
the terms $A, E_2$ and $F_2$ involving the bad boundary region $R_b$, the terms $E_3$ and $F_3$ involving $\Omega_i$ (and thus $\upsilon$), and, finally, the term $G_1$ which shall be dealt with using the weak boundedness
property. Also, when we sum over $R$ we have to use different kinds of strategies involving simple randomization, the tangent martingale trick and a certain improvement of the contraction principle. In some cases control is gained
only after using the a priori boundedness of $T$, and in these cases it is essential to get a small constant in front so that these may later be absorbed. In addition, the terms with the bad boundary regions
require that we average over all the dyadic grids too.

Let us now do all this carefully. Using the weak boundedness property holding for balls and the facts that $\Lambda B \subset \Delta$ for every $B \in \mathcal{B}$ and
$\#\mathcal{B} \lesssim C(\epsilon, \upsilon)$, we obtain that $G_1 = \alpha_{\Delta} \mu(\Delta)$, where $|\alpha_{\Delta}| \lesssim C(\epsilon, \upsilon)$. Using randomization,
Hölder's inequality and the contraction principle, we obtain (denoting the dyadic parent of $Q$ by $\tilde Q$ and similarly for $R$) that
\begin{align*}
\Big| \sum_R &B_R G_1(R) A_Q \Big| \\
&= \Big| \int_\Omega \int_X \sum_R \epsilon_R \chi_R B_R \sum_Q \epsilon_Q \alpha_{\Delta} A_Q \chi_Q \,d\mu\,d\mathbb{P}\Big|\\
& \le \Big\|\sum_R \epsilon_R \chi_R B_R\Big\|_{L^{p'}(\Omega \times X, Y^*)} \Big\| \sum_Q \epsilon_Q \alpha_{\Delta} A_Q \chi_Q\Big\|_{L^p(\Omega \times X,Y)} \\
&\lesssim C(\epsilon, \upsilon) \Big\|\sum_R \epsilon_R \psi_{\tilde R} \langle g,\psi_{\tilde R}\rangle  \Big\|_{L^{p'}(\Omega \times X, Y^*)}
\Big\| \sum_Q \epsilon_Q  \varphi_{\tilde Q} \langle \varphi_{\tilde Q}, f\rangle \Big\|_{L^p(\Omega \times X,Y)} \\
&\lesssim C(\epsilon, \upsilon) \|g\|_{L^{p'}(X, Y^*)}\|f\|_{L^p(X, Y)}.
\end{align*}

We then deal with the terms for which the summation over $R$ can be handled using this same simple randomization trick (the estimates for the corresponding parts of the matrix element are, of course, different).
One of these terms is $G_2$. We obtain using the first kernel estimate, the doubling property of $\lambda$, the separation of the different balls $B$ and $B'$ and the fact that $B, B' \subset \Delta$
that $|G_2| \lesssim C(\epsilon, \upsilon)\mu(\Delta)$. Also, we have using the a priori boundedness of $T$ and the fact that $\mu(\Omega_i) \le \upsilon\mu(\Delta)$
that $|E_3| \lesssim \upsilon^{1/p'}\|T\| \mu(\Delta)$ and $|F_3| \lesssim \upsilon^{1/p}\|T\|\mu(\Delta)$. Using the above randomization estimate then readily yields that
\begin{displaymath}
\Big| \sum_R B_R G_2(R) A_Q \Big| \lesssim C(\epsilon,\upsilon)\|g\|_{L^{p'}(X, Y^*)}\|f\|_{L^p(X, Y)}
\end{displaymath}
and
\begin{displaymath}
\Big| \sum_R B_R [E_3(R) + F_3(R)] A_Q \Big| \lesssim (\upsilon^{1/p} + \upsilon^{1/p'})\|T\|\|g\|_{L^{p'}(X, Y^*)}\|f\|_{L^p(X, Y)}.
\end{displaymath}

We now deal with the rest of the terms having separation (we already dealt with $G_2$). Namely, let us estimate the terms $B$ and $D$. However, these are so similar that we only explicitly handle $B$ here.
The first kernel estimate yields
\begin{displaymath}
|B| = |\langle \chi_{R_s} b_2, T(b_1\chi_Q) \rangle| \lesssim \int_{R_s} \int_Q \frac{1}{\lambda(y,d(x,y))}\,d\mu(y)\,d\mu(x).
\end{displaymath}
Then we note that $\lambda(y, d(x,y)) \ge \lambda(y, d(x,Q)) \ge \lambda(y, \epsilon \ell(Q)) \gtrsim \epsilon^d \lambda(y, \ell(Q))$. Thus, we may write
\begin{displaymath}
B = \beta_Q \frac{\mu(Q)\mu(R)}{\inf_{y \in Q} \lambda(y, \ell(Q))},
\end{displaymath}
where $|\beta_Q| \lesssim \epsilon^{-d}$ (note that the infimum may be zero only if $\mu(Q) = 0$). Now we may write
\begin{align*}
\sum_R B_RB(R)A_Q &= \sum_R \langle g,\psi_{\tilde R}\rangle \langle \psi_{\tilde R} \rangle_R \beta_Q \frac{\mu(Q)\mu(R)}{\inf_{y \in Q} \lambda(y, \ell(Q))} \langle \varphi_{\tilde Q} \rangle_Q \langle \varphi_{\tilde Q}, f\rangle  \\
&= \sum_R \langle g,\psi_{\tilde R}\rangle \|\psi_{\tilde R}\|_{L^1(\mu)} \frac{\tilde \beta_Q}{\inf_{y \in Q} \lambda(y, \ell(Q))} \|\varphi_{\tilde Q}\|_{L^1(\mu)} \langle \varphi_{\tilde Q}, f\rangle,
\end{align*}
where $|\tilde \beta_Q| \le |\beta_Q| \lesssim \epsilon^{-d}$. Recall that these parents $\tilde R$ and $\tilde Q$ are again good cubes. Also recall that every cube has at most $\lesssim 1$ children. So it remains to study
the series
\begin{displaymath}
\sum_R \langle g,\psi_R \rangle \|\psi_R\|_{L^1(\mu)} \frac{\sigma_Q}{\inf_{y \in Q} \lambda(y, \ell(Q))} \|\varphi_Q\|_{L^1(\mu)} \langle \varphi_Q, f\rangle,
\end{displaymath}
where again $|\sigma_Q| \lesssim \epsilon^{-d}$ (note that $\lambda(y, \ell(\tilde Q)) \lesssim \lambda(y, \ell(Q))$). Using a randomization trick and then reindexing the summation we see that this may be dominated
by $\|g\|_{L^{p'}(X,Y^*)}$ multiplied with
\begin{displaymath}
\Big\| \sum_{k \in \Z} \epsilon_k \sum_{S \in \mathcal{D}_k} \mathop{\sum_{Q \in \mathcal{D}_{\textrm{good}, k + 2r}}}_{Q \subset S}
\|\psi_R\|_{L^1(\mu)} \psi_R(x) \frac{\sigma_Q}{\inf_{y \in Q} \lambda(y, \ell(Q))} \|\varphi_Q\|_{L^1(\mu)} \langle\varphi_Q, f\rangle\Big\|_{L^p(\Omega \times X, Y)}.
\end{displaymath}
Since $R$ is good, $\ell(R) \le \delta^{-r}\ell(Q) = \delta^r \ell(S)$ and $CC_0\ell(R) > d(Q,R)$, one easily checks that $R \subset S$ (if $r$ is large enough). We then set for $S \in \mathcal{D}_k$ that
\begin{displaymath}
K_S(x,y) = \epsilon^d \mathop{\sum_{Q \in \mathcal{D}_{\textrm{good}, k + 2r}}}_{Q \subset S} \|\psi_R\|_{L^1(\mu)} \psi_R(x) \frac{\mu(S)}{\inf_{w \in Q} \lambda(w, \ell(Q))}\sigma_Q \|\varphi_Q\|_{L^1(\mu)} \varphi_Q(y) b_1(y),
\end{displaymath}
and note that the previous majorant can now be written in the form
\begin{displaymath}
\epsilon^{-d}\Big\| \sum_{k \in \Z} \epsilon_k \sum_{S \in \mathcal{D}_k} \frac{\chi_S(x)}{\mu(S)} \int_S K_S(x,y) \frac{\chi_S(y)\Delta^{b_1}_{k+2r}f(y)}{b_1(y)}\,d\mu(y)\Big\|_{L^p(\Omega \times X, Y)},
\end{displaymath}
which is amenable to the tangent martingale trick as is next demonstrated. Indeed, just note that $K_S$ is supported on $S \times S$ and that $|K_S(x,y)| \lesssim 1$ holds, and  then
divide the summation over $k$ into $\lesssim 1$ appropriate pieces to get that
\begin{displaymath}
\Big| \sum_R B_RB(R)A_Q \Big| \lesssim \epsilon^{-d}\|g\|_{L^{p'}(X,Y^*)}\|f\|_{L^p(X,Y)}.
\end{displaymath}
The same, as already stated earlier, works with $B$ replaced by $D$.

It still remains to deal with the terms involving bad boundary regions. The small term in front of $\|T\|$ is gained only after averaging over the dyadic grids $\mathcal{D}$ and $\mathcal{D}'$.
Somewhat tediously we have six ($A$, $C$, $E_1$, $E_2$, $F_1$ and $F_2$) kind of similar terms to deal with. We only deal with the term $E_1 = \langle \chi_{\Omega_Q} b_2, T(b_1\chi_{\Delta}) \rangle$
-- we chose this term as it shares the additional (albeit small) difficulty with the
term $F_2$ (not present in the four other cases) that the bad boundary region part is in some sense in the unnatural slot (here $\Omega_Q \subset Q_b$ is in the slot with $b_2$). What is useful here is that
everything is inside $\Delta$ anyway.

We turn to the details. Using randomization, Hölder's inequality and the a priori boundedness of $T$ one gets
\begin{displaymath}
\Big| \sum_R B_RE_1(R)A_Q \Big| \le \|T\| \Big\| \sum_R \epsilon_R B_R \chi_{\Omega_{Q(R)}} b_2 \Big\|_{L^{p'}(\Omega \times X, Y^*)} \Big\| \sum_Q \epsilon_Q A_Q b_1  \chi_{\Delta}\Big\|_{L^p(\Omega \times X, Y)}. 
\end{displaymath}
Now the second term is easily seen to be dominated by $\|f\|_{L^p(X,Y)}$ using the contraction principle and unconditionality.

The first term is more involved since it is here that the small factor needs to be extracted. Let us define
\begin{displaymath}
\delta(k) = \bigcup_{j=k-2r}^{k+2r} \bigcup_{R \in \mathcal{D}_j'} \delta_R.
\end{displaymath}
Note that if gen$(R) = k$, then gen$(Q(R)) \in [k-r, k+r]$, and so we must have $\chi_{\Omega_{Q(R)}} = \chi_{\Omega_{Q(R)}}\chi_{\delta(k)}\chi_R$ (recall that $\Omega_{Q(R)} \subset \Delta \subset R$).
Throwing $\chi_{\Omega_{Q(R)}}$ and $b_2$ away using the contraction principle, we get
\begin{displaymath}
\Big\| \sum_R \epsilon_R B_R \chi_{\Omega_{Q(R)}} b_2 \Big\|_{L^{p'}(\Omega \times X, Y^*)} \lesssim \Big\| \sum_{k \in \Z} \epsilon_k\chi_{\delta(k)}
\sum_{R \in \mathcal{D}'_k} B_R \chi_R\Big\|_{L^{p'}(\Omega \times X, Y^*)}.
\end{displaymath}
Now, keeping everything else fixed, we take the conditional expectation of this over the grids $\mathcal{D}'$. Using Jensen's inequality and Fubini's theorem, we get
\begin{align*}
\mathbb{E} \Big\| &\sum_{k \in \Z} \epsilon_k\chi_{\delta(k)} \sum_{R \in \mathcal{D}'_k} B_R \chi_R\Big\|_{L^{p'}(\Omega \times X, Y^*)} \\
&\lesssim \Big( \int_X \mathbb{E} \Big\| \sum_{k \in \Z} \epsilon_k\chi_{\delta(k)}(x) \sum_{R \in \mathcal{D}'_k} B_R \chi_R(x)\Big\|_{L^{p'}(\Omega, Y^*)}^{p'} \,d\mu(x) \Big)^{1/p'}.
\end{align*}
In order to gain access to a certain improvement of the contraction principle (to be formulated shortly), it is still beneficial to further dominate this by
\begin{displaymath}
\Big( \int_X \Big[\mathbb{E} \Big\| \sum_{k \in \Z} \epsilon_k\chi_{\delta(k)}(x) \sum_{R \in \mathcal{D}'_k} B_R \chi_R(x)\Big\|_{L^{p'}(\Omega, Y^*)}^t\Big]^{p'/t} \,d\mu(x) \Big)^{1/p'},
\end{displaymath}
where $t \ge p'$. We now fix $t$ once and for all demanding only that it is larger than $p$, $p'$, the cotype of $Y$ and the cotype of $Y^*$ (recall that the dual of a UMD space is UMD and that
a UMD space has nontrivial cotype). The requirements involving $p$ and the cotype of $Y$ are only needed when handling some of the other similar terms.

We now formulate the contraction principle we need (this is \cite[Lemma 3.1]{HV}).
\begin{prop}
Suppose $Z$ is a Banach space of cotype $s \in [2, \infty)$, $\xi_j \in Z$, $s < u < \infty$ and $\theta_j \in L^u(\tilde \Omega)$ (here $\tilde \Omega$ is just some probability space). Then
\begin{displaymath}
\Big\| \sum_{j=1}^{\infty} \epsilon_j\theta_j\xi_j \Big\|_{L^u(\tilde \Omega, L^2(\Omega, Z))} \lesssim \sup_j \|\theta_j\|_{L^u(\tilde \Omega)} \Big\|\sum_{j=1}^{\infty} \epsilon_j \xi_j\Big\|_{L^2(\Omega, Z)}.
\end{displaymath}
\end{prop}
Utilizing the above contraction principle together with Lemma \ref{boundarylemma} and Kahane's inequality gives (here the $L^t$ norm is taken over the probability space used
in the randomization of $\mathcal{D}'$)
\begin{align*}
\mathbb{E} \Big\|& \sum_R \epsilon_R B_R \chi_{\Omega_{Q(R)}} b_2 \Big\|_{L^{p'}(\Omega \times X, Y^*)} \\
&\lesssim \Big( \int_X  \sup_{k \in \Z} \|\chi_{\delta(k)}(x)\|_{L^t}^{p'} \Big\| \sum_{k \in \Z} \epsilon_k \sum_{R \in \mathcal{D}'_k} B_R \chi_R(x)\Big\|_{L^{p'}(\Omega, Y^*)}^{p'} \,d\mu(x) \Big)^{1/p'} \\
&\lesssim \epsilon^{\eta/t} \Big\| \sum_R \epsilon_R B_R \chi_R\Big \|_{L^{p'}(\Omega \times X, Y^*)} \\
&\lesssim \epsilon^{\eta/t} \|g\|_{L^{p'}(X, Y^*)}.
\end{align*}

We now formulate the above considerations as a proposition.
\begin{prop}
Let $\epsilon > 0$ and $\upsilon \in (0,1)$.
We have the estimate
\begin{align*}
\mathbb{E} \Big| \sum_{R \in \mathcal{D}'_{\textrm{good}}}& \mathop{\sum_{Q \in \mathcal{D}_{\textrm{good}}:\, \ell(Q) \sim \ell(R)}}_{d(Q, R) < CC_0\min(\ell(Q), \ell(R))}
\langle g, \psi_R \rangle T_{RQ} \langle \varphi_Q, f \rangle \Big| \\
&\lesssim C(\epsilon, \upsilon)\|g\|_{L^{p'}(X,Y^*)}\|f\|_{L^p(X,Y)} \\
&+ \|T\|c(\epsilon, \upsilon)\|g\|_{L^{p'}(X,Y^*)}\|f\|_{L^p(X,Y)},
\end{align*}
where we average over all the random quantities used in the randomization of the cubes, and $c(\epsilon, \upsilon)$ can be made arbitrarily small by choosing $\epsilon$ and $\upsilon$ small enough.
\end{prop}
\begin{rem}
Recall that when we dealt with the separated cubes in Proposition \ref{sepcubesprop} we had the assumption that the adapted Haar functions related to the smaller cubes are cancellative. Note that
there are only boundedly many terms with $\ell(Q) = \ell(R) = \delta^m$ where the contrary can happen (due to the assumptions about the supports of the functions $f$ and $g$). Thus, the relevant arguments
involving separated sets used in the present chapter let us also remove this assumption.
\end{rem}
\section{Completion of the proof}
Combining all that we have done in the previous sections shows that
\begin{displaymath}
\mathbb{E} \Big|\mathop{\sum_{Q \in \mathcal{D}_{\textrm{good}},\, R \in \mathcal{D}_{\textrm{good}}'}}_{\delta^{k_0} < \ell(Q), \, \ell(R) \le \delta^m} \sum_{u,v} \langle \varphi^{b_2}_{R,v}, g\rangle
\langle b_2\varphi^{b_2}_{R,v}, T(b_1\varphi^{b_1}_{Q,u})\rangle \langle \varphi^{b_1}_{Q,u} f\rangle\Big| \lesssim C(\epsilon, \upsilon) + c(\epsilon, \upsilon)\|T\|,
\end{displaymath}
where $c(\epsilon, \upsilon) \to 0$ when $\epsilon \to 0$ and $\upsilon \to 0$. Recalling (\ref{normest}) the estimate $\|T\| \lesssim 1$ follows by taking $\epsilon$ and $\upsilon$ small enough.
We have proved what we set out to prove, namely Theorem \ref{maintheorem}.
\bibliographystyle{amsalpha}%{ams-pln}%{plain}%
\bibliography{metric}
\end{document}